\documentclass[11pt]{article}

\usepackage[utf8]{inputenc}
\usepackage[T1]{fontenc}

\usepackage[verbose=true,
letterpaper,
textheight=8.7in,
textwidth=6.2in,
margin=0.9in,
headheight=12pt,
headsep=25pt,
footskip=30pt]{geometry}

\usepackage{microtype}
\usepackage{mathtools}
\usepackage{booktabs}
\usepackage[numbers,sort&compress]{natbib}
\usepackage{graphicx}
\usepackage{amsmath,amssymb,amsfonts,amsxtra,bm}
\usepackage{amsthm}
\usepackage{xspace}
\usepackage{enumerate}
\usepackage{hyperref}
\usepackage{url}

\usepackage{colortbl}

\usepackage[linesnumbered,ruled,vlined]{algorithm2e}
\usepackage{comment}
\usepackage{cleveref}

\usepackage{makecell,multirow}       
\usepackage{nicefrac}       
\usepackage{thmtools}  
\usepackage{xspace}
\usepackage{footnote, tablefootnote}
\usepackage{float}
\floatstyle{plaintop}
\restylefloat{table}
\usepackage{epstopdf}
\usepackage{latexsym}
\usepackage{graphicx}

\numberwithin{equation}{section}
\Crefname{assumption}{Assumption}{Assumptions}
\crefname{equation}{}{}
\Crefname{lemma}{Lemma}{Lemmas}
\Crefname{definition}{Definition}{Definitions}
\Crefname{proposition}{Proposition}{Propositions}
\Crefname{corollary}{Corollary}{Corollaries}
\Crefname{example}{Example}{Examples}

\usepackage{ss}       
\newtheorem{theorem}{Theorem}[section]
\newtheorem{lemma}[theorem]{Lemma}
\newtheorem{proposition}[theorem]{Proposition}
\newtheorem{corollary}[theorem]{Corollary}
\theoremstyle{definition}
\newtheorem{definition}[theorem]{Definition}
\newtheorem{remark}[theorem]{Remark}

\DeclareMathOperator{\argmin}{argmin}

\newcommand{\half}{\tfrac{1}{2}}

\renewenvironment{proof}{\paragraph{Proof:}}{\hfill$\square$}
\newenvironment{proof_sketch}{\noindent\textbf{\textit{Proof sketch}:}}{\hfill$\square$}

\begin{document}
\title{Understanding Riemannian Acceleration via a Proximal Extragradient Framework}
\author{\name{Jikai Jin} \email{jkjin@pku.edu.cn}\\ \addr{School of Mathematical Sciences, Peking University, Beijing, China}\\[4pt]
\name{Suvrit Sra} \email{suvrit@mit.edu}\\ \addr{LIDS, Massachusetts Institute of Technology, Cambridge, MA, USA}}

\maketitle

\begin{abstract}%
  We contribute to advancing the understanding of Riemannian accelerated gradient methods. In particular, we revisit ``\emph{Accelerated Hybrid Proximal Extragradient}'' (A-HPE), a powerful framework for obtaining Euclidean accelerated methods~\citep{monteiro2013accelerated}. Building on A-HPE, we then propose and analyze Riemannian A-HPE. The core of our analysis consists of two key components: (i) a set of new insights into Euclidean A-HPE itself; and (ii) a careful control of metric distortion caused by Riemannian geometry. We illustrate our framework by obtaining a few existing and new Riemannian accelerated gradient methods as special cases, while characterizing their acceleration as corollaries of our main results.
\end{abstract}

\vspace{-15pt}
\section{Introduction}\vspace*{-4pt}
\label{intro}

Convexity admits an elegant generalization beyond vector spaces to geodesic metric spaces. There, through the lens of \emph{geodesic convexity} one obtains a rich class of tractable nonconvex optimization problems, which makes the study of geodesically convex optimization potentially of far-reaching value. 
Main examples where geodesically convex optimization has been studied include certain Riemannian manifolds \citep{rapcsak1991geodesic,udriste2013convex,boumal2022intro,sra2015conic,wiesel2012geodesic}, Hadamard spaces \citep{bacak2014convex}, and non-commutative groups~\citep{burgisser2019towards}. 

The interest in geodesic convexity is paralleled by the development of optimization algorithms. Early works prove convergence for Riemannian proximal-point~\citep{ferreira2002proximal,de2016new} and Riemannian analogs of many other Euclidean methods~\citep{rapcsak1991geodesic,smith1994optimization,absil2009optimization}, though these works in general do not exploit geodesic convexity, and limit their analyses to asymptotic results. The work~\citep{zhang2016first} is the first to provide non-asymptotic rates (and iteration complexity) of first-order methods for geodesically convex optimization on Hadamard manifolds. Subsequent works establish iteration complexities for other optimization methods on Riemannian manifolds, such as variance-reduced methods~\citep{zhang2016riemannian,sato2019riemannian}, adaptive gradient methods~\citep{kasai2019riemannian}, Newton-type methods~\citep{hu2018adaptive,agarwal2021adaptive}, among many others.

A key open question is whether it is possible to develop accelerated gradient methods on Riemannian manifolds. \citet{zhang2018estimate} develop the first such method, and they show that the method achieves acceleration in a small neighborhood of the global minimum.  Later, \citet{ahn2020nesterov} show that a Riemannian version of Nesterov's method converges globally, at a rate strictly faster than gradient descent and \emph{eventually attains full acceleration}, which is defined as follows:
\begin{definition}
\label{def_eventual_acc}
We say that a gradient-based method \textit{eventually achieves acceleration} if for optimizing an $L$-smooth, $\mu$-geodesically strongly-convex function $f$, it outputs a sequence $\{w_k\}_{k \geq 1}$ with computational complexity $\mathcal{O}(k)$ that satisfies
\begin{equation}
  \label{eq:1}
  \vspace{-4pt}
    f(w_k) - f^* = \mathcal{O}\left( (1-\tau_1)(1-\tau_2)\cdots (1-\tau_k)\right),
\end{equation}
where $\tau_k \geq c \frac{\mu}{L}$ for some constant $c > 0$, and $\lim_{k\to +\infty}\tau_k = \Omega( \sqrt{\nicefrac{\mu}{L}})$.
\end{definition}
\noindent\textbf{Motivation of this work.} The abovementioned work leads one to wonder whether we can develop methods that attain full acceleration from the start, without a ``burn in'' period. Unfortunately, recent work~\citep{hamilton2021no,criscitiello2021negative} shows that full acceleration is impossible in general, which suggests that the best we can hope for is eventual acceleration.

Despite this recent progress on lower and upper bounds characterizing Riemannian acceleration, there is a considerable gap between the study of acceleration in Euclidean space versus Riemannian manifolds.\footnote{We limit our discussion to the convex case, and refer the reader to the recent work~\citep{criscitiello2021accelerated} that studies acceleration for non-geodesically-convex problems on manifolds.} 
While numerous Euclidean accelerated methods beyond the canonical one of Nesterov have been studied, it is still unknown whether they also generalize to the Riemannian setting, and as such, a systematic understanding of Riemannian acceleration is still lacking.

This gap motivates us to study Riemannian acceleration more closely. We start by revisiting \emph{Accelerated Hybrid Proximal Extragradient} (A-HPE), a powerful framework for convex optimization~\citep{monteiro2013accelerated}. Indeed, it can be shown that Nesterov's optimal method is a special case of A-HPE; \citet{monteiro2013accelerated} also propose a second-order method A-NPE, which is a specific implementation of their framework and has complexity $\widetilde{\mathcal{O}}\left( \varepsilon^{-2/7}\right)$ for $\varepsilon$-suboptimality. A hitherto unknown property of A-HPE is that it can recover a wide range of accelerated methods that have been independently proposed in past literature, e.g., the accelerated extragradient descent method of~\citet{diakonikolas2018accelerated}, the algorithm with an extra gradient descent step in \citep[Section 4]{chen2019first}, the extra-point method of \citet{huang2021unifying}, among others.

The A-HPE framework also has implications beyond usual first-order methods. A-NPE is used to design optimal second-order method in \citep{arjevani2019oracle}, and more generally, a number of works \citep{bubeck2019near,jiang2019optimal} show that A-HPE can also induce optimal higher-order methods for smooth convex functions. \citet{carmon2020acceleration} considers a different setting where one has access to a ball oracle, and they show that combining A-HPE with line search yields an accelerated method that is near-optimal. Moreover, A-HPE was extended to strongly-convex functions in~\citep{barre2021note,alves2021variants}.

\vspace*{-2pt}
\paragraph{Overview and main contributions.}
In light of the above motivation, we believe that A-HPE can help us uncover fundamental ideas behind the acceleration phenomenon. The main goal of this paper is to propose a Riemannian version of A-HPE and provide global convergence guarantees for this framework. To that end, the key contributions of our work may be summarized as follows:
\begin{list}{--}{\leftmargin=1em}\vspace*{-5pt}
  \setlength{\itemsep}{-1pt}
\item We revisit Euclidean A-HPE in \Cref{section_euclidean}, and propose to view it as the linear coupling of two approximate proximal point iterates. This viewpoint produces a simple, new analysis of A-HPE.
\item We introduce Riemannian A-HPE in \Cref{section_riemann}, which we analyze by following our Euclidean approach,  while localizing the challenges posed by Riemannian geometry. Specifically, we discover that besides the metric distortion that appears in previous works~\citep{zhang2018estimate,ahn2020nesterov}, there is an \emph{additional distortion} that must be controlled.
\item In \Cref{sec_without_distortion}, we first consider the case without additional distortion, for which we prove global eventual acceleration that not only generalizes~\citep{ahn2020nesterov}, but also offers global guarantees for some other Riemannian counterparts of Euclidean first-order accelerated methods.
\item In \Cref{sec_with_addis}, we then tackle the general case with additional distortion, which we handle by leveraging geometric bounds on Riemannian manifolds. For this general case, we obtain new local acceleration results akin to~\citep{zhang2018estimate}.
\item Finally, in \Cref{sec_1st_order}, we discuss a number of accelerated first-order methods as special cases. 
\end{list}

\paragraph{Notation and terminology.} Throughout the paper $\left\langle \cdot,\cdot\right\rangle$ denotes inner product in a Euclidean space, and $\|\cdot\|$ its induced norm. For a closed convex set $\mathcal{X} \in\mathbb{R}^d$, we define the projection $\mathcal{P}_{\mathcal{X}}(x) := \mathop{\argmin}_{y\in\mathcal{X}}\|x-y\|$. For a convex function $f:\mathbb{R}^d \to\mathbb{R}$, the \emph{proximal mapping} of $f$ is given by $\mathtt{prox}_{f}(x) := \mathop{\arg\min}_{u\in\mathbb{R}^d}\ f(u)+\half\|u-x\|^2$. 
For a $\mu$-strongly convex function $f:\mathbb{R}^d \to\mathbb{R}$, we define the quadratic function $f_w(x) := f(w)+\left\langle x-w,\nabla\right\rangle+\tfrac{\mu}{2}\|x-w\|^2$
for $w\in\mathbb{R}^d$ and $\nabla\in\partial f(w)$, so that $f_w(x)\leq f(x)$ for all $x$.

\section{A New Analysis of Euclidean A-HPE}
\label{section_euclidean}
We now revisit the Euclidean A-HPE framework---see \Cref{AHPE_Euclidean}. We propose to analyze A-HPE via the proximal point method, leading to a novel analysis that is simpler and more intuitive (in our opinion) than previous approaches~\citep{barre2021note,alves2021variants}. More importantly, this analysis helps us develop Riemannian A-HPE, our main focus. 

Throughout this section we assume that $f$ is $\mu$-strongly-convex. Our description follows~\citep{barre2021note} and relies on the key concept of \emph{inexact proximal operators}. Our definition below is equivalent to the one in~\citep[Definition 2.3]{barre2021note} that relies on the primal-dual gap of a proximal function. We use our version for ease of analysis. \Cref{eps_subgrad_appendix} provides additional intuition on this concept by relating it to $\varepsilon$-subgradients.

\begin{algorithm}[t]
\SetKwInOut{KIN}{Input}
\caption{Accelerated hybrid proximal extragradient (A-HPE) method}
\label{AHPE_Euclidean}
\KIN{Objective function $f$, initial point $x_0$, step size $\lambda_k>0,k=1,2,\cdots$, a sequence $\left\{\sigma_k\right\}$ in $[0,1]$, initial weight $A_0 \geq 0$}
$z_0 = x_0$ \\
\For{$k=1,2,\cdots$}{
$a_{k+1} \gets\frac{\left(1+2 \mu A_{k}\right) \lambda_{k+1}+\sqrt{\left(1+2 \mu A_{k}\right)^{2} \lambda_{k+1}^{2}+4\left(1+\mu A_{k}\right) A_{k} \lambda_{k+1}}}{2}$ \\
$A_{k+1} \gets A_k+a_{k+1}$\\
$y_{k} \gets x_{k}+\frac{a_{k+1}\left(1+\mu A_{k}\right)}{A_{k+1}+\mu \left( a_{k+1}A_k+A_k A_{k+1} \right)}\left(z_{k}-x_{k}\right)$ \\
$\varepsilon_{k} \gets \frac{\sigma_{k}^{2}}{2\left(1+\lambda_{k} \mu\right)^{2}}\left\|x_{k+1}-y_{k}\right\|^{2}$ \\
choose $\left(x_{k+1}, v_{k+1}\right) \in \mathtt{iprox}_{f}(y_k,\lambda_k,\varepsilon_k)$\\
$z_{k+1} \gets z_{k}+\frac{a_{k+1}}{1+\mu A_{k+1}}\left(\mu\left(x_{k+1}-z_{k}\right)-v_{k+1}\right)$
}
\end{algorithm}

\begin{definition}
\label{iprox}
~\citep[Lemma 2.4]{barre2021note} We write $(x,v) \in \mathtt{iprox}_{f}(y,\lambda,\varepsilon)$ 
if
\begin{equation}
\label{iprox_ineq}
    \frac{1}{2(1+\lambda\mu)^2}\|x-y+\lambda v\|^2+\frac{\lambda}{1+\lambda\mu}\bigl( f(x)-f(w)-\left\langle x-w,v\right\rangle + \frac{\mu}{2}\|x-w\|^2 \bigr) \leq \varepsilon,
\end{equation}
where $w\in\mathbb{R}^d$ satisfies $v-\mu x+\mu w \in\partial f(w)$ and $\varepsilon \ge 0$.
\end{definition}
\noindent If $\varepsilon=0$ and $w=x$, then $v \in \partial f(x)$ and $x+\lambda v=y$, which recovers the \emph{exact} proximal operator.

With Definition~\ref{iprox} in hand, up to the specification of the sequences $\{\lambda_k\}$ and $\{\sigma_k\}$, all steps of \Cref{AHPE_Euclidean} are implementable. Hence, we are ready to state the main result of this section.

\begin{theorem}
\label{euclidean_convergence}
For the iterates produced by Algorithm \ref{AHPE_Euclidean}, we have the function suboptimality bound
\begin{equation*}
    f(x_k)-f(x^*) \leq \frac{A_0(f(x_0)-f(x^*))+\frac{1+\mu A_0}{2}\|x_0-x^*\|^2 }{A_k}
    = \mathcal{O}\bigl( \Pi_{i=1}^{k-1}\bigl(1+\max\bigl\{ \mu\lambda_i,\sqrt{\mu\lambda_i}\bigr\}\bigr)^{-1}\bigr).
\end{equation*}
\end{theorem}
Assuming that $f$ is also $L$-smooth, a number of first-order methods (including Nesterov's  method) can be considered as a special cases of \Cref{AHPE_Euclidean}, with the choice $\lambda_i = \mathcal{O}(\nicefrac{1}{L})$. \Cref{euclidean_convergence} then implies that these methods have the optimal convergence rate of $\mathcal{O}\big( (1+\sqrt{\nicefrac{\mu}{L}})^{-k}\bigr)$. We do not present concrete examples here since this is not our main focus, but we will include detailed discussions about such special cases for the Riemannian setting later in the paper. 

\subsection{Overview of the proof of Theorem~\ref{euclidean_convergence}}
\label{overview}
We now overview our proof technique for \Cref{euclidean_convergence}, which sheds light on the specific parameter choices and updates that comprise \Cref{AHPE_Euclidean}. To aid exposition, we trade simplicity for rigor in our overview below, and defer a fully rigorous proof to \Cref{euclid_appendix}.

Motivated by~\citep{allen2014linear,ahn2020proximal}, we view \Cref{AHPE_Euclidean} as a combination of two approximate PPM (proximal point method) updates, each using a different notion of approximation. The first uses the inexact proximal operator from~\Cref{iprox}, while the second arises from minimizing a quadratic lower-approximation of $f$. When properly combined, these two steps allow us to prove the following theorem that immediately implies \Cref{euclidean_convergence}.
\begin{theorem}
\label{euclid_potential_decrease}
The potential function $p_k := f(x_k) + \frac{1+\mu A_k}{2} \|z_k-x^*\|^2$ is decreasing for all $k\ge 0$. 
\end{theorem}

The potential function in \Cref{euclid_potential_decrease} has two terms: the objective $f(x_k)$ and a distance term involving $\|z_k-x^*\|^2$. We analyze these terms separately; they are associated with the two approximate PPM steps alluded to above, and the amount they change with $k$ must be carefully combined to ensure $p_k \ge p_{k+1}$. We start with \Cref{informal_descent} to bound the change in function value.
\begin{lemma}
\label{informal_descent}
Denote $\nabla_{k+1} = v_{k+1}+\mu\left( w_{k+1}-x_{k+1}\right) \in\partial f(w_{k+1})$; when $\varepsilon_k$ is small, we have
\begin{subequations}\label{f_value_descent}
    \begin{align}
    f(x_{k+1}) &\lesssim f(w_{k+1})+\tfrac{1}{2\mu}\left( \|v_{k+1}\|^2-\|\nabla_{k+1}\|^2 \right) \label{f_value_descent1} \\
    &\leq f(x_k)-\tfrac{\mu}{2}\|x_k-x_{k+1}+\mu^{-1}v_{k+1}\|^2+\tfrac{1}{2\mu}\|v_{k+1}\|^2. \label{f_value_descent2}
    \end{align}
\end{subequations}
\end{lemma}

The proof of \Cref{informal_descent} is given in \Cref{euclid_appendix}. Inequality~\eqref{f_value_descent2} is not exact; we omit an additional term that depends on $\varepsilon_k$ for ease of presentation. Inequality~\eqref{f_value_descent} can be understood as a \emph{descent inequality} for the function value at $x_k$, albeit with an \emph{error term} $\|v_{k+1}\|$. When this term is large, we may no longer be able to control the change in function values.

Next, we bound the change in the distance term. Observe that Line~8 of \Cref{AHPE_Euclidean} is nothing but $z_{k+1} \gets \frac{1+\mu A_k}{1+\mu A_{k+1}} z_k + \frac{\mu a_{k+1}}{1+\mu A_{k+1}}\left(w_{k+1}-\mu^{-1}\nabla_{k+1}\right) = \mathop{\argmin}_{z} \{ f_{w_{k+1}}(z) + \frac{1+\mu A_k}{2a_{k+1}}\|z-z_k\|^2\}$, where $f_{w_{k+1}}(z) := f(w_{k+1})+\left\langle \nabla_{k+1},z-w_{k+1}\right\rangle +\frac{\mu}{2}\|z-w_{k+1}\|^2$ is a lower quadratic approximation of $f$. \Cref{informal_dist} then helps us bound this change.
\begin{lemma}[Approximate distance change]
\label{informal_dist}
When $\varepsilon_k$ is small, we have
\begin{subequations}\label{dist_descent}
    \begin{align} 
     &\quad \tfrac{1+\mu A_k}{2}\|z_k-x^*\|^2-\tfrac{1+\mu A_{k+1}}{2}\|z_{k+1}-x^*\|^2 \nonumber \\
     &\geq a_{k+1}(f(w_{k+1})-f(x^*)) + \tfrac{\mu a_{k+1}(1+\mu A_k)}{2(1+\mu A_{k+1})}\|z_k-w_{k+1}+\mu^{-1}\nabla_{k+1}\|^2 - \tfrac{a_{k+1}}{2\mu}\left\| \nabla_{k+1} \right\|^2  \label{use_lemma}  \\
    &\gtrsim a_{k+1}(f(x_{k+1})-f(x^*)) 
    + \tfrac{\mu a_{k+1}(1+\mu A_k)}{2(1+\mu A_{k+1})}\|z_k-x_{k+1}+\mu^{-1}v_{k+1}\|^2 
    - \tfrac{a_{k+1}}{2\mu}\left\| v_{k+1} \right\|^2. \label{use_iprox_def}
    \end{align}
\end{subequations}
\end{lemma}
Note that \Cref{use_lemma} is very similar to the prox-grad inequality~\citep[Theorem 10.16]{beck2017first} and the fundamental inequality of mirror descent~\citep[Section 2.2]{allen2014linear} that imply \emph{contraction of distance} with a proximal iteration. Inequality~\eqref{use_iprox_def} is not exact since it depends on $\varepsilon_k$. Again, the term $\|v_{k+1}\|$ prevents us from directly deducing contraction of the distance to $x^*$.

The inequalities in \Cref{informal_descent} and \Cref{informal_dist} reveal a challenge faced when proving descent of the potential function: we must control the magnitude of $v_{k+1}$. Specifically, consider the situation where the positive terms $\|x_k-x_{k+1}+\mu^{-1}v_{k+1}\|$ in (\ref{f_value_descent2}) and $\|z_k-x_{k+1}+\mu^{-1}v_{k+1}\|$ in (\ref{use_iprox_def}) are small but $\|v_{k+1}\|$ is large. But when $\varepsilon_k$ is small, \Cref{iprox} also implies that $x_{k+1}-y_k \approx -\lambda_k v_{k+1}$, which further implies that $y_k-x_{k+1}+\mu^{-1}v_{k+1} \approx \left( \mu^{-1}+\lambda_k \right)v_{k+1}$ is large. This observation suggests that a contradiction is arrived at if we choose $y_k$ on the line segment connecting $x_k$ and $z_k$, i.e., $y_k=\tau x_k+(1-\tau)z_k$. Why? Since in this case, if $y_k-x_{k+1}+\mu^{-1}v_{k+1}$ is large, then we can directly deduce that (a convex combination) of the terms $\|x_k-x_{k+1}+\mu^{-1}v_{k+1}\|$ and $\|z_k-x_{k+1}+\mu^{-1}v_{k+1}\|$ is large using Cauchy-Schwarz. Remarkably, this argument suggests that we should choose $\tau$ such that $\tau:1-\tau$ is equal to ratio of the coefficients of $\|x_k-x_{k+1}+\mu^{-1}v_{k+1}\|^2$ and $\|z_k-x_{k+1}+\mu^{-1}v_{k+1}\|^2$. Therefore we can use these terms to cancel out the error induced by $v_{k+1}$, and ultimately attain the desired potential function descent, leading to Theorem~\ref{euclid_potential_decrease}.


\section{From Euclidean to Riemannian A-HPE}
\label{section_riemann}
We are now ready to generalize Euclidean A-HPE to the Riemannian setting (more precisely, to Hadamard manifolds). In \Cref{preliminary_Riemann} we recall key notation for the Riemannian setting, and  \Cref{Riemann_analysis} is dedicated to the analysis of our proposed framework, Riemannian A-HPE.

\subsection{Riemannian preliminaries and notation}
\label{preliminary_Riemann}
We refer the readers to standard textbooks, e.g., \citep{lee2006riemannian,jost2008riemannian} for an in-depth introduction; we recall below key notation and concepts. 

A smooth manifold $\mathcal{M}$ is called a \emph{Riemannian manifold} if an inner product $\left\langle \cdot , \cdot \right\rangle_x$ is defined in the tangent space $T_{x}\mathcal{M}$ for all $x\in\mathcal{M}$ and the inner product varies smoothly in $x$. In this section, we use the notation $\left\langle \cdot , \cdot \right\rangle$ and omit the dependence on $x$, since it is clear from the context. We define $\| \cdot\|$ to be the norm induced by the inner product i.e., $\|v\|:=\sqrt{\left\langle v,v \right\rangle}$.

A curve on $\mathcal{M}$ is called a \emph{geodesic} if it is locally distance-minimizing. The \emph{exponential map}, denoted by $\mathtt{Exp}_x$, maps a vector $v \in T_{x}\mathcal{M}$ to a point $y \in\mathcal{M}$ such that there exists a geodesic $\gamma :[0,1] \to \mathcal{M}$ such that $\gamma(0)=x$, $\gamma(1)=y$ and $\gamma'(0) = v$. We assume that the sectional curvature of $\mathcal{M}$ is non-positive and lower bounded by $-K$, where $K$ is a positive real number. Under this assumption, any two points on $\mathcal{M}$ are connected by a unique geodesic, and thus the \emph{inverse exponential map} $\mathtt{Exp}_{x}^{-1}: \mathcal{M} \to T_{x}\mathcal{M}$ is well-defined. 

We use $d(x,y)$ to denote the \emph{Riemannian distance} between $x$ and $y$. 
The definition of exponential map implies that $d(x,y)=\left\| \mathtt{Exp}_{x}^{-1}(y) \right\|$. We will also use the \emph{tangent space distance}: $d_{w}(x,y):=\| \mathtt{Exp}_{w}^{-1}(x)-\mathtt{Exp}_{w}^{-1}(y)\|$. Note that $d_w(x,y) \leq d(x,y)$ for all $w,x,y \in\mathcal{M}$.
We say that a function $f:\mathcal{M}\to\mathbb{R}$ is $\mu$-geodesically-strongly-convex for $\mu\geq 0$, if for any $x\in\mathcal{M}$ there exists a non-empty set $\partial f(x)$, such that for all $y \in\mathcal{M}$ and $v \in\partial f(x)$ we have
\begin{equation}
    \notag
    f(y)\geq f(x)+\langle v,\mathtt{Exp}_{x}^{-1}(y)\rangle + \tfrac{\mu}{2} d^2(x,y).
\end{equation}
Thus $f_x(y) := f(x)+\langle v,\mathtt{Exp}_{x}^{-1}(y)\rangle + \frac{\mu}{2}d^2(x,y)$ is a lower approximation of $f$. 

We use $\Gamma_{x}^{y}$ to denote the \emph{parallel transport} from $T_{x}\mathcal{M}$ to $T_{y}\mathcal{M}$ along the geodesic connecting $x$ and $y$. Using parallel transport, we can define a natural generalization of $L$-smoothness to the Riemannian setting. 
We say that $f:\mathcal{M}\to\mathbb{R}$ is \emph{$L$-smooth} if for all $x,y \in\mathcal{M}$, we have 
\begin{equation}
    \left\| \Gamma_{x}^y \nabla f(x)- \nabla f(y) \right\| \leq L\cdot d(x,y).
\end{equation}

\subsection{The proposed Riemannian A-HPE framework}
\label{Riemann_analysis}
In this section, we first present a straightforward generalization of Euclidean A-HPE (\Cref{AHPE_Euclidean}) to the Riemannian setting---see \Cref{AHPE_Riemann}. Then, we introduce a number results useful in its convergence analysis. Our presentation largely follows the Euclidean setting, except for a number of new challenges posed by Riemannian geometry. Throughout, we assume that $f$ is $\mu$-geodesically strongly convex, and that the sectional curvature of $\mathcal{M}$ lies in $[-K,0]$.
\begin{algorithm}
\SetKwInOut{KIN}{Input}
\caption{\textbf{Riemannian} accelerated hybrid proximal extragradient method}
\label{AHPE_Riemann}
\KIN{Objective function $f$, initial point $x_0$, `reference' step size $\lambda>0$,  $\sigma_k\in(0,1)$ 
and initial weights $A_0,B_0 \geq 0$}
$z_0 \gets x_0$\\
\For{$k=0,1,\cdots$}{
choose a valid distortion rate $\delta_k$ according to \Cref{ahn_distortion} \\
$\theta_{k} \gets$ the smaller root of $B_k (1-\theta)^2=\mu\lambda_k\theta\left( (1-\theta)B_k+\frac{\mu}{2}\delta_k A_k \right)$\\
$B_{k+1} \gets \frac{B_k}{\theta_k\delta_k}$, $a_{k+1} \gets 2\mu^{-1}(1-\theta_k)B_{k+1}$ and $A_{k+1} \gets A_k+a_{k+1}$ \\
$y_{k} \gets \mathtt{Exp}_{x_k}\bigl( \frac{\theta_k a_{k+1}}{A_k+\theta a_{k+1}}\mathtt{Exp}_{x_k}^{-1}(z_k) \bigr)$ \\
choose $\left(x_{k+1}, v_{k+1}\right) \in \mathtt{iprox}_{f}^{w_{k+1}}(y_k,\lambda_k,\varepsilon_k)$ with $\varepsilon_{k} = \frac{\sigma_{k}^{2}}{2\left(1+\lambda_{k} \mu\right)^{2}}d_{w_{k+1}}^2(x_{k+1},y_k)$\\
$z_{k+1} \gets \mathtt{Exp}_{w_{k+1}}\bigl( (1-\theta_k)\mathtt{Exp}_{w_{k+1}}^{-1}(x_{k+1})+\theta_k\mathtt{Exp}_{w_{k+1}}^{-1}(z_k) - \frac{1-\theta_k}{\mu} v_{k+1}\bigr)$\\
}
\end{algorithm}

Beyond the natural replacement of vector space operations with their Riemannian counterparts, there are two key differences between \Cref{AHPE_Euclidean} and \Cref{AHPE_Riemann}: (i) the latter uses a Riemannian version of the \emph{iprox} operator; and (ii) it uses additional  parameters ($B_k$ and $\delta_k$) in its updates. 
We define the Riemannian \emph{iprox} operator as follows.
\begin{definition}[Riemannian inexact proximal operator]
\label{Riemannian_iprox}
    For $x,y, w \in\mathcal{M}$, $v\in T_w \mathcal{M}$ and $\lambda,\varepsilon \geq 0$, we write $(x,v) \in \mathtt{iprox}_{f}^{w}(y,\lambda,\varepsilon)$
    if we have the inequality
    \begin{equation}
    \label{Riemann_iprox}
        \frac{\|\mathtt{Exp}_w^{-1}(x)-\mathtt{Exp}_w^{-1}(y)+\lambda v\|^2}{2(1+\lambda\mu)^2} 
        +\frac{\lambda \bigl( f(x)-f(w)-\left\langle \mathtt{Exp}_w^{-1}(x),v \right\rangle + \frac{\mu}{2}d^2(x,w) \bigr)}{1+\lambda\mu} \leq \varepsilon,
    \end{equation}
    and $v-\mu\mathtt{Exp}_w^{-1}(x) \in\partial f(w)$.
\end{definition}
Key among the additional parameters is $\delta_k$, the \emph{distortion rate} that is used to model the non-linearity of the exponential map. The concept of distortion rate is not new, and was introduced in~\citep{ahn2020nesterov} to analyze potential function decrease. We will formally define it in \Cref{def_distortion_rate}. By setting $\delta_k = 0$ and $B_k = \frac{1+\mu A_k}{2}$, \Cref{AHPE_Riemann} recovers \Cref{AHPE_Euclidean} in the Euclidean setting. Before discussing technical details, let us give an informal statement of our main result. In subsequent sections, we sketch its proof and provide the formal statements, while full, rigorous proofs are deferred to~\Cref{riem_appendix}.

\begin{theorem}[informal version of~\Cref{simple_convergence_thm} and \Cref{informal_convergence_general}]
\label{informal}
Under mild conditions on the choice of $w_{k+1}$, for $\mu$-strongly convex and $L$-smooth function $f$, the following statements hold:
\begin{enumerate}[(1).]
    \item Suppose $w_{k+1}$ lies on the geodesic between $x_k$ and $z_k$, then the iterates $\{x_k\}$ generated by \Cref{AHPE_Riemann} eventually achieve acceleration (\emph{cf.}\ \Cref{def_eventual_acc}) with arbitrary initialization.
    \item In the general case, the iterates $\{x_k\}$ achieve acceleration as long as the initialization is in a $\mathcal{O}\bigl( K^{-\nicefrac{1}{2}}\bigl(\nicefrac{\mu}{L}\bigr)^{\nicefrac{3}{4}}\bigr)$ neighbourhood of $x^*$.
\end{enumerate}
\end{theorem}
As we will see later, the Riemannian analogs of Nesterov's method considered in \citep{zhang2018estimate} and \citep{ahn2020nesterov} belong to the first case in \Cref{informal}, and thus, follow as corollaries of our main result. We will also discuss additional instances of each case in \Cref{sec_1st_order}.

\subsection{Potential Function Analysis for Riemannian A-HPE}
Similar to the Euclidean setting, we define the potential function
\begin{equation}
  \label{eq:2}
  p_k = A_k\cdot (f(x_k)-f(x^*))+B_k\cdot d_{w_k}^2(z_k,x^*).
\end{equation}
Note that in the above definiton, we use the tangent space distance $d_{w_k}$ rather than the Riemannian distance $d$. Indeed, when generalizing our analysis to the Riemannian setting, we need to work with vectors in tangent spaces, so that it is more convenient to use the tangent space distance here.

Our Euclidean analysis is based on a separate analysis of two approximate PPM schemes, one leading to \Cref{informal_descent} for function value, and the other leading to \Cref{informal_dist} for distance to $x^*$. While it is straightforward to generalize \Cref{informal_descent} to the Riemannian setting by using strong-convexity and \Cref{Riemannian_iprox}, it is hard to bound the distance term $B_k\cdot d_{w_k}^2(z_k,x^*)-B_{k+1}\cdot d_{w_{k+1}}^2(z_{k+1},x^*)$ because it involves vectors in two different tangent spaces $T_{w_k}\mathcal{M}$ and $T_{w_{k+1}}\mathcal{M}$. 
Taking cue from~\citep{ahn2020nesterov}, we also use the notion of \emph{distortion rates} to overcome this issue.

\begin{definition}
\label{def_distortion_rate}
We say that $\delta_k>0$ is a \emph{valid distortion rate} if $d_{w_{k+1}}^2(z_k,x^*) \leq \delta_k d_{w_k}^2(z_k,x^*)$.
\end{definition}

To be able to use valid distortion rates in an actual algorithm, it is crucial to avoid dependence on the unknown optimal point $x^*$. To that end, the next lemma shows that one can obtain a valid distortion rate in terms of $d(w_k,z_k)$ instead. 
\begin{lemma}[\protect{\citep[Lemma 4.1]{ahn2020nesterov}}]
\label{ahn_distortion} For any points $x,y,z \in\mathcal{M}$, we have $d^2(x,y) \leq T_{K}(d(x,z))d_z^2(x,y)$, where the function $T_K(\cdot)$ is defined as
\begin{equation}
\notag
T_{K}(r):= \begin{cases}\max \bigl\{1+4\bigl(\frac{\sqrt{K} r}{\tanh (\sqrt{K} r)}-1\bigr),\bigl(\frac{\sinh (2 \sqrt{K} \cdot r)}{2 \sqrt{K} \cdot r}\bigr)^{2}\bigr\}, & \text { if } r>0, \\ 1, & \text { if } r=0.\end{cases}
\end{equation}
In particular, $\delta_k = T_{K}(d(w_k,z_k))$ is a valid distortion rate.
\end{lemma}
\noindent Assuming access to valid distortion rates, we can obtain the Riemannian analog to \Cref{informal_dist}.

\begin{lemma}
\label{reference_distortion}
Suppose that $\delta_k>0$ is a valid distortion rate, and $B_{k+1} = \frac{B_k}{\theta_k\delta_k}$, then 
\begin{equation}
    \notag
    \begin{aligned}
    B_{k}d_{w_k}^2(z_k,x^*)-B_{k+1}d_{w_{k+1}}^2(z_{k+1},x^*)
    &\geq (1-\theta_k)B_{k+1}\bigl(\tfrac{2}{\mu}(f(w_{k+1})-f(x^*))-\tfrac{1}{\mu^2}\|\nabla_{k+1}\|^2  \\
    &\quad  + \theta_k\bigl\| \mathtt{Exp}_{w_{k+1}}^{-1}(z_k)-\mathtt{Exp}_{w_{k+1}}^{-1}(x_{k+1})+\mu^{-1}v_{k+1} \bigr\|^2\bigr) .
    \end{aligned}
\end{equation}
\end{lemma}

Now, it remains to combine the two PPM schemes (based on function value and distance to $x^*$) to obtain a bound for the potential function. Specifically, we can prove the following key lemma.
\begin{lemma}
\label{update_distortion}
Let $a_{k+1}= A_{k+1}-A_k= \frac{2}{\mu}(1-\theta_k)B_{k+1}$, and $p_k$ be given by~\eqref{eq:2}. Then,
\begin{equation}
\label{update_distortion_eq}
    \begin{aligned}
    p_k-p_{k+1} &\geq \tfrac{\mu}{2}(\theta_k a_{k+1}+A_k)\bigl\| \mathtt{Exp}_{w_{k+1}}^{-1}(y_k')-\mathtt{Exp}_{w_{k+1}}^{-1}(x_{k+1})+\mu^{-1}v_{k+1} \bigr\|^2 \\
    &+\tfrac{A_{k+1}}{2\lambda_k\sigma_k}\bigl\| \mathtt{Exp}_{w_{k+1}}^{-1}(y_k)-\mathtt{Exp}_{w_{k+1}}^{-1}(x_{k+1})-\lambda_k v_{k+1} \bigr\|^2 \\
    &+ \tfrac{\mu\theta_{k}a_{k+1}A_k}{2(A_k+\theta_{k}a_{k+1})}d_{w_{k+1}}^2(x_k,z_k) - \tfrac{\sigma_k A_{k+1}}{2\lambda_k}d_{w_{k+1}}^2(x_{k+1},y_k)-\tfrac{A_{k+1}}{2\mu}\|v_{k+1}\|^2,
    \end{aligned}
\end{equation}
where 
\begin{equation}
    \label{def_yk_prime}
    y_k' = \mathtt{Exp}_{w_{k+1}}\bigl( \tfrac{A_k}{A_k+\theta_{k}a_{k+1}}\mathtt{Exp}_{w_{k+1}}^{-1}(x_k)+\tfrac{\theta_{k}a_{k+1}}{A_k+\theta_{k}a_{k+1}}\mathtt{Exp}_{w_{k+1}}^{-1}(z_k) \bigr).
\end{equation}
\end{lemma}

The main feature of \Cref{update_distortion} is the presence of a new point $y_k'$ in  \Cref{update_distortion_eq}. In the Euclidean setting, we combined the two approximate PPM schemes by choosing $y_k$ on the line segment between $x_k$ and $z_k$. Generalizing this update rule to the Riemannian setting, we naturally choose $y_k$ on the geodesic connecting $x_k$ and $z_k$. However, here a subtle complication arises: since we are working with vectors in the tangent space $\mathcal{T}_{w_{k+1}}$, what we really want is that $\mathtt{Exp}_{w_{k+1}}^{-1}(y_k)$ be a convex combination of $\mathtt{Exp}_{w_{k+1}}^{-1}(x_k)$  and $\mathtt{Exp}_{w_{k+1}}^{-1}(z_k)$. This subtlety explains why $y_k'$ as defined by~\eqref{def_yk_prime} appears in the bound~\eqref{update_distortion_eq}. Further, note that since $y_k'$ depends on $w_{k+1}$, it \emph{cannot} be used as the update rule of $y_k$.

In Euclidean A-HPE, $y_k'$ does not complicate matters since we always have $y_k=y_k'$. However, $y_k \neq y_k'$ in general for the Riemannian setting, which prevents us from mimicking the Euclidean analysis. Indeed, \Cref{update_distortion} highlights an \emph{additional distortion} that arises for Riemannian A-HPE and is not present in previous works that focus on Nesterov's method~\citep{zhang2018estimate,ahn2020nesterov}. Indeed, the algorithms analyzed in these previous works are a special case of the general A-HPE framework, where the particular specialization of the updates bypasses the additional distortion that arises more generally. We expand on these observations below.

\subsection{Basic A-HPE: convergence without additional distortion}
\label{sec_without_distortion}
We first consider the special case where $y_k = y_k'$. This equality holds as long as $w_{k+1}$ is chosen on the geodesic connecting $x_k$ and $z_k$. In this case, we can derive potential decrease from \Cref{update_distortion} by using an analysis similar to the Euclidean setting. Doing so, we obtain the following main result regarding the convergence rate of Riemannian A-HPE.
\begin{theorem}
\label{simple_convergence_thm}
Suppose in \Cref{AHPE_Riemann}, we choose $\lambda_k = \lambda$ and $w_{k+1}$ lies on the geodesic connecting $x_k$ and $z_k$ such that $d(w_{k+1},y_k) = \mathcal{O}(1)$, then we have
\begin{equation}
    \label{convergence_ineq}
    \begin{aligned}
    f(x_k)-f(x^*) &\leq \nicefrac{p_0}{A_k},\quad\text{and}\quad
    \lim_{k\to +\infty}\nicefrac{A_{k+1}}{A_k} = 1+\mu\lambda + \sqrt{\mu\lambda(1+\mu\lambda)}.
    \end{aligned}
\end{equation}
\end{theorem}
\begin{proof_sketch}
The first inequality follows directly from potential decrease. Define $\xi_k = \nicefrac{a_k}{A_k}$, then it suffices to show that $\lim_{k\to +\infty} \xi_k = \sqrt{\frac{\mu\lambda}{1+\mu\lambda}}$. The proof relies on the following recursive equation:
\begin{equation}
  \label{recursive_xi_main}
  \delta_k\xi_{k+1}\bigl(\xi_{k+1}-\nicefrac{\mu\lambda}{1+\mu\lambda}\bigr) = \xi_k^2(1-\xi_{k+1}).
\end{equation}
If $\delta_k = 1$, then we can show that $\{\xi_k\}$ converges to the fixed point of \Cref{recursive_xi_main}, which is $\sqrt{\frac{\mu\lambda}{1+\mu\lambda}}$. In our setting, $\delta_k$ is not constant, but potential decrease implies that $\delta_k-1$ converges to $0$ at a linear rate. Therefore, we can still obtain the desired result.
\end{proof_sketch}

The assumption $d(w_{k+1},y_k) = \mathcal{O}(1)$ ensures that the iterates of \Cref{AHPE_Riemann} are uniformly bounded; otherwise the distortion error can become arbitrarily large. This assumption is trivially true when $w_{k+1}=y_k$, which holds for a number first-order methods that we will discuss in \Cref{sec_1st_order}. \Cref{simple_convergence_thm} also immediately implies the following result, which plays a crucial role when studying first-order methods as special cases of \Cref{AHPE_Riemann}.
\begin{corollary}
\label{asymp_acc}
Suppose that $f$ is $L$-smooth and $\lambda_k=\lambda = \Theta\left(1/L\right)$. Under the conditions in \Cref{simple_convergence_thm}, \Cref{AHPE_Riemann} eventually achieves acceleration.
\end{corollary}

\noindent Finally, we bound the number of iterations sufficient for achieving full acceleration.
\begin{theorem}
\label{achieve_acc_time}
Under the assumptions in \Cref{simple_convergence_thm}, if $f$ is $L$-smooth and $\lambda_k=\lambda=\mathcal{O}(\nicefrac{1}{L})$, then we have $\xi_k \geq \frac{1}{2}\sqrt{\frac{\mu\lambda}{1+\mu\lambda}}$ after $T = \widetilde{\mathcal{O}}( \nicefrac{L}{\mu})$ iterations, where $\widetilde{\mathcal{O}}$ hides logarithmic terms. As a result, \Cref{AHPE_Riemann} achieves acceleration in at most $\widetilde{\mathcal{O}}( \nicefrac{L}{\mu})$ iterations.
\end{theorem}

\vspace*{-8pt}
\subsection{The general case of A-HPE: handling additional distortion}
\label{sec_with_addis}
In general, we do not have $y_k = y_k'$, so that an \emph{additional distortion} appears in the analysis of Riemannian A-HPE. To overcome the challenge posed by this distortion, we take an approach that is based on deriving an upper bound for the tangent space distance between $y_k$ and $y_k'$. We need the following lemma, which is a variant of ~\citep[Section B.3]{sun2019escaping}.

\begin{lemma}
\label{contraction}
~\citep[Lemma 3]{sun2019escaping} Let $x\in\mathcal{M}$ and $y, a\in T_{x}\mathcal{M}$. Let $z = \mathtt{Exp}_x(a)$, then
\begin{equation}
\notag
d\left(\mathtt{Exp}_{x}(y+a), \mathtt{Exp}_{z}\left(\Gamma_{x}^{z} y\right)\right) \leq   \min \{\|a\|,\|y\|\} S_{K}(\|a\|+\|y\|),
\end{equation}
where $S_{K}(r) =  \cosh(\sqrt{K}r)-\nicefrac{\sinh\left(\sqrt{K}r\right)}{\sqrt{K}r}$.
\end{lemma}

Note that a key feature of the function $S_{K}$ is that $\lim_{r\to 0}S_{K}(r)=0$. By using \Cref{contraction}, we can obtain an upper bound on $d_{w_{k+1}}(y_k,y_k')$ in terms of $S_K(\cdot)$ and a distance term.

\begin{lemma}
\label{addis_bound}
We have for all $k \geq 1$ that 
\begin{equation}
    \label{addis_bound_eq}
    \begin{aligned}
    d_{w_{k+1}}(y_k,y_k') \leq 2 d^{*}(w_{k+1};x_k,z_k)\cdot S_{K}\left( d(x_k,z_k)+d^*(w_{k+1};x_k,z_k) \right),
    \end{aligned}
\end{equation}
where $d^*(w;x,z) := \min\left\{ d(w,y) \mid y = \mathtt{Exp}_{x}(t\cdot\mathtt{Exp}_{x}^{-1}(z)), t \in [0,1] \right\}$ is the distance from $w$ to the geodesic connecting $x$ and $z$.
\end{lemma}

When $d_{w_{k+1}}^2(y_k,y_k')$ is small, we can imagine that the algorithm still behaves similar to the $y_k = y_k'$ case studied in \Cref{sec_without_distortion}. From a technical standpoint, to lower-bound the potential difference $p_k - p_{k+1}$, the key difference between the Riemannian setting with the Euclidean setting is the presence of an additional negative term that depends on $d_{w_{k+1}}^2(y_k,y_k')$. As a result, potential decrease can still be guaranteed if the RHS of \Cref{addis_bound_eq} is smaller than the positive terms. This yields our main result for the convergence of \Cref{AHPE_Riemann} in the general case, as stated below.

\begin{theorem}[informal]
\label{informal_convergence_general}
Suppose that $f$ is $L$-smooth, $\sigma_k =\sigma \in (0,1)$ and $\lambda_k=\lambda = \mathcal{O}(\nicefrac{1}{L})$. Under regularity conditions on the choice of $w_{k+1}$, if the initialization satisfies $d(x_0,x^*) =\mathcal{O}\bigl(K^{-\nicefrac{1}{2}}(\nicefrac{\mu}{L})^{\nicefrac{3}{4}}\bigr)$ and $B_0 = \frac{\mu}{2}A_0 > 0$, then potential decrease holds, and $\xi_k := \frac{a_k}{A_k} = \Theta\bigl(\sqrt{\frac{\mu\lambda}{1+\mu\lambda}}\bigr)$. 
\end{theorem}
\begin{proof_sketch}
  The proof is by induction on $k$. When $k=0$, by using regularity conditions on $w_1$, we can derive an upper bound for the RHS of \Cref{addis_bound_eq}, which implies potential decrease.
Now suppose that potential decrease holds for $k$. By the definition of the potential function $p_k$, we can show that $d^2(x_k,x^*)$ and $d^2(z_k,x^*) = \mathcal{O}(K^{-1}(\nicefrac{\mu}{L})^{\nicefrac{1}{2}})$.

Note that the distortion rate $\delta_k \leq 1 + \mathcal{O}\left( K d^2(w_k,z_k)\right)$, and under regularity conditions on $w_k$ we can bound $\delta_k$ by $1+\mathcal{O}( \sqrt{\nicefrac{\mu}{L}})$; $\xi_{k+1}$ can then be lower-bounded using the recursive equation \Cref{recursive_xi} and the lower bound for $\xi_k$. Finally, the RHS of \eqref{addis_bound_eq} can be directly upper-bounded using the bounds for $x_k$, $z_k$ and regularity conditions on $w_{k+1}$, which implies potential decrease for $k+1$.
\end{proof_sketch}

\vskip6pt
The regularity conditions on the sequence $\{w_k\}$ are described formally in  \Cref{convergence_general}, and they play a crucial role in \Cref{informal_convergence_general}. In short, they require that $\{w_k\}$ is not too far away from the sequence $\{x_k\}$ and $\{z_k\}$, since otherwise the algorithm may suffer from large distortion error.

\Cref{informal_convergence_general} implies that as long as the initialization is inside a $\mathcal{O}( K^{-\nicefrac{1}{2}}(\nicefrac{\mu}{L})^{\nicefrac{3}{4}})$ neighbourhood of the global minimum $x^*$, then it can achieve the accelerated rate. 

\begin{corollary}
\label{cor_acc}
  Under the assumptions of \Cref{informal_convergence_general}, we have
  \begin{equation}
      \notag
      f(x_k)-f(x^*) \leq c_1 K^{-1}L(\nicefrac{\mu}{L})^{\frac{3}{2}} \cdot\bigl( 1- c_2\sqrt{\nicefrac{\mu}{L}}\bigr)^{k},
  \end{equation}
  for some numerical constants $c_1, c_2 > 0$.
\end{corollary}

\section{Special cases of Riemannian A-HPE: acceleration of several first-order methods}
\label{sec_1st_order}
Inspired by Nesterov's method, a number of different accelerated methods have been proposed in the Euclidean setting \citep{diakonikolas2018accelerated,chen2019first,huang2021unifying}. These methods are empirically observed to be superior in some aspects (e.g., robustness to noise, possibly smaller constants in convergence bounds, etc.) However, they are derived using a variety of very different techniques, which obscures their common origin. In contrast, we observe that all of them can be deduced from A-HPE quite naturally and straightforwardly.

At the same time, in the Riemannian setting only a generalized version of Nesterov's method is known to achieve acceleration~\citep{zhang2018estimate,ahn2020nesterov}. Can we design other accelerated methods, similar to those in the Euclidean setting? The answer is ``yes,'' and we discuss below several special cases obtained from our Riemannian A-HPE framework. We divide these special cases into two categories: (i) those without additional distortion (\Cref{sec_without_distortion}), and which eventually attain acceleration with arbitrary initialization due to \Cref{simple_convergence_thm}; and (ii) those that can suffer additional distortion studied in~\Cref{sec_with_addis}, for which local acceleration is ensured by \Cref{informal_convergence_general}. Detailed derivations of the methods studied in this section are given in \Cref{1st_order_appendix}.

\subsection{Accelerated methods without additional distortion}
\textbf{Riemannian Nesterov's method. } Nesterov's method has a direct generalization to the Riemannian setting, as proposed and analyzed in \citep{zhang2018estimate,ahn2020nesterov}; it takes the form:
\begin{equation}
\label{Riemann_nesterov}
    \begin{aligned}
    y_{k} &= \mathtt{Exp}_{x_k}\bigl( \tfrac{\theta_k a_{k+1}}{A_k+\theta a_{k+1}}\mathtt{Exp}_{x_k}^{-1}(z_k) \bigr), \\
    x_{k+1} &= \mathtt{Exp}_{y_k}(-\lambda\nabla f(y_k)), \\
    z_{k+1} &= \mathtt{Exp}_{y_k}( \theta_k\mathtt{Exp}_{y_k}^{-1}(z_k) - \mu^{-1}(1-\theta_k) \nabla f(y_k)).
    \end{aligned}
\end{equation}
\noindent We can derive this algorithm from \Cref{AHPE_Riemann} by choosing $w_{k+1} = y_k$, $x_{k+1} = \mathtt{Exp}_{y_k}(-\lambda_k\nabla f(y_k))$ and $v_{k+1}=\nabla f(y_k)+ \mu\mathtt{Exp}_{y_k}^{-1}(x_{k+1})$. Additional distortion is not present since $w_{k+1}=y_k$. We also recover the result of \citet{ahn2020nesterov} that~\eqref{Riemann_nesterov} can eventually achieve acceleration; the local acceleration result of \citet{zhang2018estimate} can also be directly deduced from \Cref{informal_convergence_general}.

\vskip6pt
\noindent\textbf{Riemannian Nesterov's method with multiple gradient steps.} We can also perform multiple gradient descent (GD) steps from $y_k$ to obtain $x_{k+1}$. \citet[Algorithm 3]{chen2019first} present a method of this type in the Euclidean setting. Here we consider a Riemannian version of their method:
\begin{equation}
    \label{Riemann_adGD}
    \begin{aligned}
    y_{k} &= \mathtt{Exp}_{x_k}\bigl( \tfrac{\theta_k a_{k+1}}{A_k+\theta a_{k+1}}\mathtt{Exp}_{x_k}^{-1}(z_k) \bigr), \\
    \tilde{x}_{k+1} &= \mathtt{Exp}_{y_k}(-\lambda\nabla f(y_k)),\\
    x_{k+1} &= \mathtt{Exp}_{\tilde{x}_{k+1}}(-\lambda\nabla f(\tilde{x}_{k+1})), \\
    z_{k+1} &= \mathtt{Exp}_{y_k}\bigl( \theta_k\mathtt{Exp}_{y_k}^{-1}(z_k) - \mu^{-1}(1-\theta_k) \nabla f(y_k) \bigr).
    \end{aligned}
\end{equation}
Method~\eqref{Riemann_adGD} can be derived from \Cref{AHPE_Riemann} by choosing $x_{k+1}$ as the result of two GD steps; the other variables the same as Riemannian Nesterov's method.

\subsection{Accelerated methods with additional distortion}
\textbf{Riemannian accelerated extra-gradient descent (RAXGD).} We consider a Riemannian version of the accelerated extra-gradient method (AXGD) proposed by \citet{diakonikolas2018accelerated}: 
\begin{equation}
    \label{RAXGD}
    \begin{aligned}
    y_{k} &= \mathtt{Exp}_{x_k}\bigl( \tfrac{\theta_k a_{k+1}}{A_k+\theta a_{k+1}}\mathtt{Exp}_{x_k}^{-1}(z_k) \bigr), \\
    x_{k+1} &= \mathtt{Exp}_{y_k}(-\lambda\nabla f(y_k)),\\
    z_{k+1} &= \mathtt{Exp}_{x_{k+1}}\bigl( \theta_k\mathtt{Exp}_{x_{k+1}}^{-1}(z_k) - \mu^{-1}(1-\theta_k) \nabla f(x_{k+1})\bigr).
    \end{aligned}
\end{equation}
Method~\eqref{RAXGD} can be recovered from \Cref{AHPE_Riemann} by choosing $v = \nabla f(x_{k+1})$, and $w_{k+1} = x_{k+1} = \mathtt{Exp}_{y_k}(-\lambda_k\nabla f(y_k))$. While \citet{diakonikolas2018accelerated} obtain AXGD via a specifically chosen discretization of  suitable continuous-time dynamics, we observe that (R)AXGD can be deduced from A-HPE quite straightforwardly.

\vskip6pt
\noindent\textbf{Generalized RAXGD.} We can deduce from \Cref{AHPE_Riemann} a generalized
version of RAXGD. Specifically, we replace the gradient descent step of $x_{k+1}$ in \Cref{RAXGD} with the following:
\begin{equation}
  \label{eq:3}
    w_{k+1} = \mathtt{Exp}_{y_k}(-\lambda_k\nabla f(y_k)),\quad d(x_{k+1},w_{k+1}) \leq c \cdot d(x_{k+1},y_k).
\end{equation}
Here $c \in (0,1)$ is a numerical constant. Please refer to \Cref{1st_order_appendix} for more details and discussion of the choices proposed in~\eqref{eq:3}.

\vskip6pt
\noindent\textbf{The extra-point framework of \citet{huang2021unifying}.} Recently, a general framework was proposed by \citet{huang2021unifying} for obtaining accelerated methods in the Euclidean setting (\emph{cf}.\ eq.,~(26) therein). We observe that their framework has a natural interpretation via the PPM viewpoint discussed in \Cref{section_euclidean}, though upon using a less general version of update rules compared with A-HPE. A detailed comparison between their framework and A-HPE is provided in \Cref{comparison}, where we also present a Riemannian generalization of their algorithm. Using our approach of analyzing Riemannian A-HPE, local acceleration can be shown for the resulting algorithm, while for a special case (corresponding to the algorithm described in~\citep[eq.(38)]{huang2021unifying}), global eventual acceleration can also be achieved.

\section{Conclusion and future directions}
In this paper, we propose an alternative viewpoint of the Euclidean A-HPE framework of~\citep{monteiro2013accelerated} via the proximal point method. This viewpoint allows us to derive a simple and novel convergence analysis of A-HPE; it also plays a pivotal role in obtaining \Cref{AHPE_Riemann}, our proposed  generalization of A-HPE to the Riemannian setting. While most of our Euclidean proof generalizes to the Riemannian setting, there is an additional distortion caused by the non-linearity of the exponential map that we must overcome; we model this distortion by leveraging geometric tools to complete the convergence analysis. Our main results include local acceleration of Riemannian A-HPE in its most general form, which we sharpen to global (eventual) acceleration whenever additional distortion is not present. We demonstrate the generality of our framework by discussing several accelerated first-order methods as special cases, recovering the recent results~\citep{zhang2018estimate,ahn2020nesterov} as special cases, obtaining Riemannian counterparts of other accelerated (Euclidean) algorithms, and deriving new algorithms from our framework. 

An aspect more basic worth noting is that this work also contributes toward a more thorough understanding of accelerated methods on Riemannian manifolds. Even on Euclidean spaces, our PPM-based approach may be of independent interest, since it provides a unified way for analyzing several accelerated methods that have been proposed in the literature and analyzed using a number of different techniques. Nonetheless, there are some important questions that remain unanswered.

First, we only show local convergence in the general case where additional distortion arises. It is unclear whether Riemannian A-HPE can indeed fail to converge in some cases, or whether the locality restriction is a shortcoming of our analysis. Nevertheless, we believe that some regularization conditions on the specification of the \emph{iprox} operator (e.g., the conditions in \Cref{convergence_general}) are necessary, since large distortion error would unavoidably impact the rate of convergence.

Second, in this paper we focus on accelerated first-order methods for strongly-convex functions on non-positively curved manifolds. The main challenge of the convex setting is that the effect of metric distortion would not asymptotically vanish as in the strong-convex setting. For manifolds with positive curvature, it is necessary to restrict the iterates inside a convex set, for example by using projection operators, but this may hurt the analysis of acceleration. Also, as discussed in \Cref{intro}, the A-HPE framework can also lead to optimal higher-order methods in Euclidean setting. However, to the best of our knowledge, optimal higher-order methods and their convergence rates are not known in the Riemannian setting. It may be useful (and feasible) to design such methods based on the Riemannian A-HPE framework introduced in this paper. 

Finally, a broader goal in the study of acceleration is to develop theory and algorithms for non-Euclidean settings beyond those offered by Riemannian geometry.      

\setlength{\bibsep}{3pt}
\bibliographystyle{plainnat}
\bibliography{references}

\begin{thebibliography}{39}
\providecommand{\natexlab}[1]{#1}
\providecommand{\url}[1]{\texttt{#1}}
\expandafter\ifx\csname urlstyle\endcsname\relax
  \providecommand{\doi}[1]{doi: #1}\else
  \providecommand{\doi}{doi: \begingroup \urlstyle{rm}\Url}\fi

\bibitem[Absil et~al.(2009)Absil, Mahony, and Sepulchre]{absil2009optimization}
P-A Absil, Robert Mahony, and Rodolphe Sepulchre.
\newblock \emph{Optimization algorithms on matrix manifolds}.
\newblock Princeton University Press, 2009.

\bibitem[Agarwal et~al.(2021)Agarwal, Boumal, Bullins, and
  Cartis]{agarwal2021adaptive}
Naman Agarwal, Nicolas Boumal, Brian Bullins, and Coralia Cartis.
\newblock Adaptive regularization with cubics on manifolds.
\newblock \emph{Mathematical Programming}, 188\penalty0 (1):\penalty0 85--134,
  2021.

\bibitem[Ahn(2020)]{ahn2020proximal}
Kwangjun Ahn.
\newblock From proximal point method to {N}esterov's acceleration.
\newblock \emph{arXiv preprint arXiv:2005.08304}, 2020.

\bibitem[Ahn and Sra(2020)]{ahn2020nesterov}
Kwangjun Ahn and Suvrit Sra.
\newblock From {N}esterov’s estimate sequence to {R}iemannian acceleration.
\newblock In \emph{Conference on Learning Theory}, pages 84--118. PMLR, 2020.

\bibitem[Allen-Zhu and Orecchia(2014)]{allen2014linear}
Zeyuan Allen-Zhu and Lorenzo Orecchia.
\newblock Linear coupling: An ultimate unification of gradient and mirror
  descent.
\newblock \emph{arXiv preprint arXiv:1407.1537}, 2014.

\bibitem[Alves(2021)]{alves2021variants}
M~Marques Alves.
\newblock Variants of the {A-HPE} and large-step {A-HPE} algorithms for
  strongly convex problems with applications to accelerated high-order tensor
  methods.
\newblock \emph{arXiv preprint arXiv:2102.02045}, 2021.

\bibitem[Arjevani et~al.(2019)Arjevani, Shamir, and Shiff]{arjevani2019oracle}
Yossi Arjevani, Ohad Shamir, and Ron Shiff.
\newblock Oracle complexity of second-order methods for smooth convex
  optimization.
\newblock \emph{Mathematical Programming}, 178\penalty0 (1):\penalty0 327--360,
  2019.

\bibitem[Bac{\'a}k(2014)]{bacak2014convex}
Miroslav Bac{\'a}k.
\newblock \emph{Convex analysis and optimization in {H}adamard spaces}.
\newblock de Gruyter, 2014.

\bibitem[Barr{\'e} et~al.(2021)Barr{\'e}, Taylor, and Bach]{barre2021note}
Mathieu Barr{\'e}, Adrien Taylor, and Francis Bach.
\newblock A note on approximate accelerated forward-backward methods with
  absolute and relative errors, and possibly strongly convex objectives.
\newblock \emph{arXiv preprint arXiv:2106.15536}, 2021.

\bibitem[Beck(2017)]{beck2017first}
Amir Beck.
\newblock \emph{First-order methods in optimization}.
\newblock SIAM, 2017.

\bibitem[Boumal(2022)]{boumal2022intro}
Nicolas Boumal.
\newblock An introduction to optimization on smooth manifolds.
\newblock To appear with Cambridge University Press, Jan 2022.
\newblock URL \url{http://www.nicolasboumal.net/book}.

\bibitem[Br{\o}ndsted and Rockafellar(1965)]{brondsted1965subdifferentiability}
Arne Br{\o}ndsted and Ralph~Tyrrell Rockafellar.
\newblock On the subdifferentiability of convex functions.
\newblock \emph{Proceedings of the American Mathematical Society}, 16\penalty0
  (4):\penalty0 605--611, 1965.

\bibitem[Bubeck et~al.(2019)Bubeck, Jiang, Lee, Li, and
  Sidford]{bubeck2019near}
S{\'e}bastien Bubeck, Qijia Jiang, Yin~Tat Lee, Yuanzhi Li, and Aaron Sidford.
\newblock Near-optimal method for highly smooth convex optimization.
\newblock In \emph{Conference on Learning Theory}, pages 492--507. PMLR, 2019.

\bibitem[B{\"u}rgisser et~al.(2019)B{\"u}rgisser, Franks, Garg, Oliveira,
  Walter, and Wigderson]{burgisser2019towards}
Peter B{\"u}rgisser, Cole Franks, Ankit Garg, Rafael Oliveira, Michael Walter,
  and Avi Wigderson.
\newblock Towards a theory of non-commutative optimization: Geodesic 1st and
  2nd order methods for moment maps and polytopes.
\newblock In \emph{2019 IEEE 60th Annual Symposium on Foundations of Computer
  Science (FOCS)}, pages 845--861. IEEE, 2019.

\bibitem[Carmon et~al.(2020)Carmon, Jambulapati, Jiang, Jin, Lee, Sidford, and
  Tian]{carmon2020acceleration}
Yair Carmon, Arun Jambulapati, Qijia Jiang, Yujia Jin, Yin~Tat Lee, Aaron
  Sidford, and Kevin Tian.
\newblock Acceleration with a ball optimization oracle.
\newblock \emph{Advances in Neural Information Processing Systems}, 33, 2020.

\bibitem[Chen and Luo(2019)]{chen2019first}
Long Chen and Hao Luo.
\newblock First order optimization methods based on {H}essian-driven {N}esterov
  accelerated gradient flow.
\newblock \emph{arXiv preprint arXiv:1912.09276}, 2019.

\bibitem[Criscitiello and
  Boumal(2021{\natexlab{a}})]{criscitiello2021accelerated}
Christopher Criscitiello and Nicolas Boumal.
\newblock An accelerated first-order method for non-convex optimization on
  manifolds.
\newblock \emph{arXiv:2008.02252}, 2021{\natexlab{a}}.

\bibitem[Criscitiello and Boumal(2021{\natexlab{b}})]{criscitiello2021negative}
Christopher Criscitiello and Nicolas Boumal.
\newblock Negative curvature obstructs acceleration for geodesically convex
  optimization, even with exact first-order oracles.
\newblock \emph{arXiv preprint arXiv:2111.13263}, 2021{\natexlab{b}}.

\bibitem[de~Carvalho~Bento et~al.(2016)de~Carvalho~Bento, da~Cruz~Neto, and
  Oliveira]{de2016new}
Glaydston de~Carvalho~Bento, Jo{\~a}o~Xavier da~Cruz~Neto, and Paulo~Roberto
  Oliveira.
\newblock A new approach to the proximal point method: convergence on general
  {R}iemannian manifolds.
\newblock \emph{Journal of Optimization Theory and Applications}, 168\penalty0
  (3):\penalty0 743--755, 2016.

\bibitem[Diakonikolas and Orecchia(2018)]{diakonikolas2018accelerated}
Jelena Diakonikolas and Lorenzo Orecchia.
\newblock Accelerated extra-gradient descent: A novel accelerated first-order
  method.
\newblock In \emph{9th Innovations in Theoretical Computer Science Conference
  (ITCS 2018)}. Schloss Dagstuhl-Leibniz-Zentrum fuer Informatik, 2018.

\bibitem[Ferreira and Oliveira(2002)]{ferreira2002proximal}
OP~Ferreira and PR~Oliveira.
\newblock Proximal point algorithm on {R}iemannian manifolds.
\newblock \emph{Optimization}, 51\penalty0 (2):\penalty0 257--270, 2002.

\bibitem[Hamilton and Moitra(2021)]{hamilton2021no}
Linus Hamilton and Ankur Moitra.
\newblock No-go theorem for acceleration in the hyperbolic plane.
\newblock \emph{arXiv preprint arXiv:2101.05657}, 2021.

\bibitem[Hu et~al.(2018)Hu, Milzarek, Wen, and Yuan]{hu2018adaptive}
Jiang Hu, Andre Milzarek, Zaiwen Wen, and Yaxiang Yuan.
\newblock Adaptive quadratically regularized newton method for {R}iemannian
  optimization.
\newblock \emph{SIAM Journal on Matrix Analysis and Applications}, 39\penalty0
  (3):\penalty0 1181--1207, 2018.

\bibitem[Huang and Zhang(2021)]{huang2021unifying}
Kevin Huang and Shuzhong Zhang.
\newblock A unifying framework of accelerated first-order approach to strongly
  monotone variational inequalities.
\newblock \emph{arXiv preprint arXiv:2103.15270}, 2021.

\bibitem[Jiang et~al.(2019)Jiang, Wang, and Zhang]{jiang2019optimal}
Bo~Jiang, Haoyue Wang, and Shuzhong Zhang.
\newblock An optimal high-order tensor method for convex optimization.
\newblock In \emph{Conference on Learning Theory}, pages 1799--1801. PMLR,
  2019.

\bibitem[Jost(2008)]{jost2008riemannian}
J{\"u}rgen Jost.
\newblock \emph{Riemannian geometry and geometric analysis}.
\newblock Springer, seventh edition, 2008.

\bibitem[Kasai et~al.(2019)Kasai, Jawanpuria, and Mishra]{kasai2019riemannian}
Hiroyuki Kasai, Pratik Jawanpuria, and Bamdev Mishra.
\newblock Riemannian adaptive stochastic gradient algorithms on matrix
  manifolds.
\newblock In \emph{International Conference on Machine Learning}, pages
  3262--3271. PMLR, 2019.

\bibitem[Lee(2006)]{lee2006riemannian}
John~M Lee.
\newblock \emph{Riemannian manifolds: an introduction to curvature}, volume
  176.
\newblock Springer Science \& Business Media, 2006.

\bibitem[Monteiro and Svaiter(2013)]{monteiro2013accelerated}
Renato~DC Monteiro and Benar~Fux Svaiter.
\newblock An accelerated hybrid proximal extragradient method for convex
  optimization and its implications to second-order methods.
\newblock \emph{SIAM Journal on Optimization}, 23\penalty0 (2):\penalty0
  1092--1125, 2013.

\bibitem[Rapcs{\'a}k(1991)]{rapcsak1991geodesic}
Tam{\'a}s Rapcs{\'a}k.
\newblock Geodesic convexity in nonlinear optimization.
\newblock \emph{Journal of Optimization Theory and Applications}, 69\penalty0
  (1):\penalty0 169--183, 1991.

\bibitem[Sato et~al.(2019)Sato, Kasai, and Mishra]{sato2019riemannian}
Hiroyuki Sato, Hiroyuki Kasai, and Bamdev Mishra.
\newblock Riemannian stochastic variance reduced gradient algorithm with
  retraction and vector transport.
\newblock \emph{SIAM Journal on Optimization}, 29\penalty0 (2):\penalty0
  1444--1472, 2019.

\bibitem[Smith(1994)]{smith1994optimization}
Steven~T Smith.
\newblock Optimization techniques on {R}iemannian manifolds.
\newblock \emph{Fields institute communications}, 3\penalty0 (3):\penalty0
  113--135, 1994.

\bibitem[Sra and Hosseini(2015)]{sra2015conic}
Suvrit Sra and Reshad Hosseini.
\newblock Conic geometric optimization on the manifold of positive definite
  matrices.
\newblock \emph{SIAM Journal on Optimization}, 25\penalty0 (1):\penalty0
  713--739, 2015.

\bibitem[Sun et~al.(2019)Sun, Flammarion, and Fazel]{sun2019escaping}
Yue Sun, Nicolas Flammarion, and Maryam Fazel.
\newblock Escaping from saddle points on {R}iemannian manifolds.
\newblock \emph{Advances in Neural Information Processing Systems},
  32:\penalty0 7276--7286, 2019.

\bibitem[Udriste(2013)]{udriste2013convex}
Constantin Udriste.
\newblock \emph{Convex functions and optimization methods on {R}iemannian
  manifolds}, volume 297.
\newblock Springer Science \& Business Media, 2013.

\bibitem[Wiesel(2012)]{wiesel2012geodesic}
Ami Wiesel.
\newblock Geodesic convexity and covariance estimation.
\newblock \emph{IEEE transactions on signal processing}, 60\penalty0
  (12):\penalty0 6182--6189, 2012.

\bibitem[Zhang and Sra(2016)]{zhang2016first}
Hongyi Zhang and Suvrit Sra.
\newblock First-order methods for geodesically convex optimization.
\newblock In \emph{Conference on Learning Theory}, pages 1617--1638. PMLR,
  2016.

\bibitem[Zhang and Sra(2018)]{zhang2018estimate}
Hongyi Zhang and Suvrit Sra.
\newblock An estimate sequence for geodesically convex optimization.
\newblock In \emph{Conference On Learning Theory}, pages 1703--1723. PMLR,
  2018.

\bibitem[Zhang et~al.(2016)Zhang, J~Reddi, and Sra]{zhang2016riemannian}
Hongyi Zhang, Sashank J~Reddi, and Suvrit Sra.
\newblock Riemannian {SVRG}: Fast stochastic optimization on {R}iemannian
  manifolds.
\newblock \emph{Advances in Neural Information Processing Systems},
  29:\penalty0 4592--4600, 2016.

\end{thebibliography}

\newpage
\appendix
\section{Connection between \textit{iprox} and $\varepsilon$-subgradient}
\label{eps_subgrad_appendix}

In this section, we show the equivalence between the \textit{iprox} operator (cf. \Cref{iprox}) and the notion of $\varepsilon$-subdifferential ~\citep[Section 3]{brondsted1965subdifferentiability}.

\begin{definition}
Suppose that $h:\mathbb{R}^d \to\mathbb{R}$ is $\mu$-strongly convex and $x\in\mathbb{R}^d$. We say that $u \in\mathbb{R}^d$ is an $\varepsilon$-subgradient of $f$ at $x$ if the inequality
\begin{equation}
    \notag
    f(y) \geq f(x)+\left\langle u,y-x\right\rangle +\frac{\mu}{2}\|y-x\|^2-\varepsilon
\end{equation}
holds for all $y\in\mathbb{R}^d$.
\end{definition}

Note that the condition $v-\mu x+\mu w \in\partial f(w)$ in \Cref{iprox} implies that $0 \in \partial\Phi(w)$, where
\begin{equation}
    \notag
    \Phi(z) = f(x)-f(z)-\left\langle x-z,v \right\rangle+\frac{\mu}{2}\|x-z\|^2
\end{equation}
Moreover, $\Phi(z)$ is concave since $f$ is $\mu$-strongly convex. Hence $w \in\mathop{\arg\max}_{z}\Phi(z)$, and for any $z$ we have
\begin{equation}
    \notag
    f(z) \geq f(x)+\left\langle z-x,v\right\rangle+\frac{\mu}{2}\|x-z\|^2 -\frac{1+\lambda\mu}{\lambda}\varepsilon.
\end{equation}
In other words, $v$ is an $\frac{1+\lambda\mu}{\lambda}\varepsilon$-subgradient of $f$ at $x$. The inequality \Cref{iprox_ineq} further implies that $x+\lambda v \approx y$. Thus \Cref{iprox} indeed defines an approximation to the exact proximal point, for which $x+\lambda v = y$ and $v \in\partial f(x)$.

\section{Details and proofs of \Cref{section_euclidean}}
\label{euclid_appendix}

We define the potential function
\begin{equation}
    \label{potential}
    p_k = A_k (f(x_k)-f(x^*))+\frac{1+\mu A_k}{2}\|z_k-x^*\|^2
\end{equation}
our goal is to show that the sequence $\{p_k\}$ is non-increasing, so that we can obtain a bound for $f(x_k)-f(x^*)$.

In the work \citep{barre2021note} the authors also use a potential function approach to show convergence of A-HPE. Motivated by our linear coupling viewpoint, we present our analysis in a clearer way, which is helpful for addressing the key challenges that may arise in the Riemannian setting.
\\

We first present a simple lemma which will be used to simplify our analysis. It can be checked using simple algebraic calculations, so we omit its proof here.

\begin{lemma}[Interpolation implies contraction]
\label{simple}
For all $p,q\in\mathbb{R}$ such that $p+q>0$, we have
\begin{equation}
\notag
    p\|x\|^2+q\|y\|^2 = (p+q) \left\|\frac{p}{p+q}x+\frac{q}{p+q}y\right\|^2+\frac{pq}{p+q}\|x-y\|^2
\end{equation}
\end{lemma}

We define $\nabla_{k+1} := v_{k+1}+\mu\left( w_{k+1}-x_{k+1}\right) \in\partial f(w_{k+1})$, so the last line of \Cref{AHPE_Euclidean} can be re-written as
\begin{equation}
    \label{alternative_update}
    z_{k+1} \gets \frac{1+\mu A_k}{1+\mu A_{k+1}} z_k + \frac{\mu a_{k+1}}{1+\mu A_{k+1}}w_{k+1}-\frac{a_{k+1}}{1+\mu A_{k+1}}\nabla_{k+1}.
\end{equation}

The following lemma deals with the squared-distance terms in the potential function.

\begin{lemma}
\label{distance_lemma}
We have
\begin{equation}
\label{sc1}
    \begin{aligned}
    &\quad \frac{1+\mu A_k}{2}\|z_k-x^*\|^2 - \frac{1+\mu A_{k+1}}{2}\|z_{k+1}-x^*\|^2 
    \geq a_{k+1}(f(w_{k+1})-f(x^*)) \\
    &\quad + \frac{\mu a_{k+1}(1+\mu A_k)}{2(1+\mu A_{k+1})}\|z_k-w_{k+1}+\mu^{-1}\nabla_{k+1}\|^2 - \frac{a_{k+1}}{2\mu}\left\| \nabla_{k+1} \right\|^2
    \end{aligned}
\end{equation}
\end{lemma}

\begin{proof}
First note that
\begin{subequations}
\label{distance_ineq}
    \begin{align}
    &\quad \frac{1+\mu A_{k+1}}{2}\|z_{k+1}-x^*\|^2 - \frac{1+\mu A_k}{2}\|z_k-x^*\|^2 \nonumber \\
    &=\frac{\mu a_{k+1}}{2}\left\| \frac{1+\mu A_{k+1}}{\mu a_{k+1}}(z_{k+1}-x^*) - \frac{1+\mu A_k}{\mu a_{k+1}}(z_k-x^*)\right\|^2 \label{distance_ineq_1st} \\
    &\quad - \frac{(1+\mu A_k)(1+\mu A_{k+1})}{2\mu a_{k+1}}\|z_{k+1}-z_k\|^2 \nonumber\\
    &= \frac{\mu a_{k+1}}{2} \| x^*-w_{k+1}+\mu^{-1}\nabla_{k+1}\|^2 - \frac{\mu a_{k+1}(1+\mu A_k)}{2(1+\mu A_{k+1})}\|z_k-x_{k+1}+\mu^{-1}v_{k+1}\|^2 \label{distance_ineq_2nd}
    \end{align}
\end{subequations}
where \Cref{simple} is used in \Cref{distance_ineq_1st}, and \Cref{distance_ineq_2nd} follows from \Cref{alternative_update}. Thus, by strong convexity of $f$ and the definition of $w_{k+1}$ (see \Cref{iprox}) we have
\begin{equation}
\notag
    \begin{aligned}
    f(x^*) &\geq f(w_{k+1}) + \left\langle \nabla_{k+1}, x^*-w_{k+1} \right\rangle + \frac{\mu}{2}\|x^*-w_{k+1}\|^2 \\
    &= f(w_{k+1})+ \frac{\mu}{2}\| x^*-w_{k+1}+\mu^{-1}\nabla_{k+1}\|^2-\frac{1}{2\mu}\left\| \nabla_{k+1} \right\|^2
    \end{aligned}
\end{equation}
so that
\begin{equation}
\notag
    \begin{aligned}
    & a_{k+1}(f(x^*)-f(w_{k+1})) \geq \frac{1+\mu A_{k+1}}{2}\|z_{k+1}-x^*\|^2 - \frac{1+\mu A_k}{2}\|z_k-x^*\|^2 \\
    &+ \frac{\mu a_{k+1}(1+\mu A_k)}{2(1+\mu A_{k+1})}\|z_k-w_{k+1}+\mu^{-1}\nabla_{k+1}\|^2 - \frac{a_{k+1}}{2\mu}\left\| \nabla_{k+1} \right\|^2
    \end{aligned}
\end{equation}
as desired.
\end{proof}

\begin{remark}
\label{distance_remark}
The derivation of \Cref{distance_ineq} reveals the connection between the choice of parameters in the update \Cref{alternative_update} and the growth of coefficient of the distance term in the construction of potential function. This observation will provide guidelines for choosing parameters in the Riemannian setting (cf. \Cref{reference_distortion}).
\end{remark}

Now it suffices to deal with the function value terms. Strong convexity implies that
\begin{equation}
\label{sc2}
    f(x_k) \geq f(w_{k+1})+ \frac{\mu}{2}\| x_k-w_{k+1}+\mu^{-1}\nabla_{k+1}\|^2-\frac{1}{2\mu}\left\| \nabla_{k+1} \right\|^2
\end{equation}
and
\begin{equation}
\label{sc3}
    f(x_{k+1}) \geq f(w_{k+1})+ \frac{\mu}{2}\| \mu^{-1}v_{k+1}\|^2-\frac{1}{2\mu}\left\| \nabla_{k+1} \right\|^2
\end{equation}
while the definition of $w_{k+1}$ implies
\begin{equation}
\label{def_w}
\begin{aligned}
&\quad \frac{\sigma_{k}^{2}}{2}\left\|x_{k+1}-y_{k}\right\|^{2} \geq \frac{1}{2}\left\|x_{k+1}-y_{k}+\lambda_{k} v_{k+1}\right\|^{2} \\
&+\lambda_{k}\left(1+\lambda_{k} \mu\right)\left(f\left(x_{k+1}\right)-f\left(w_{k+1}\right)+\frac{1}{2\mu}\left(\|\nabla_{k+1}\|^2-\|v_{k+1}\|^2\right)\right)
\end{aligned}
\end{equation}

We now seek a correct linear combination of the above inequalities to match the coefficient of $p_k-p_{k+1}$. Note that adding \eqref{sc3} and \eqref{def_w} leads to the following simpler inequality
\begin{equation}
    \label{simpler}
    \|x_{k+1}-y_k+\lambda_k v_{k+1}\|^2 \leq \sigma_k^2 \|x_{k+1}-y_k\|^2
\end{equation}

The following lemma proves non-increasing of the potential function, which is based on the above observations and results.

\begin{lemma}
\label{potential_dec}
We have for all $k\geq 0$ that
\begin{equation}
\notag
    p_k - p_{k+1} \geq \frac{\mu\lambda_k A_k(1+\mu A_k)}{2a_{k+1}}\|x_k-z_k\|^2 + \frac{(1-\sigma_k^2)A_{k+1}}{2\lambda_k}\|x_{k+1}-y_k\|^2
\end{equation}
\end{lemma}

\begin{proof}
By combining the inequalities \eqref{sc1},\eqref{sc2},\eqref{def_w} we have
\begin{equation}
    \label{long_ineq}
    \begin{aligned}
    &\quad p_k-p_{k+1} \\
    &= \underbrace{\left(\frac{1+\mu A_k}{2}\|z_k-x^*\|^2-\frac{1+A_{k+1}}{2}\|z_{k+1}-x^*\|^2+a_{k+1}\left( f(x^*)-f(w_{k+1})\right)\right)}_{\text{use }  \Cref{distance_lemma}} \\
    &\quad +\underbrace{A_k(f(x_{k})-f(w_{k+1}))}_{\text{use }\eqref{sc2}}+\underbrace{A_{k+1}(f(w_{k+1})-f(x_{k+1}))}_{\text{use }\eqref{def_w}} \\
    &\geq \frac{\mu a_{k+1}(1+\mu A_k)}{2(1+\mu A_{k+1})}\|z_k-x_{k+1}+\mu^{-1}v_{k+1}\|^2 + \frac{\mu A_k}{2}\|x_k-x_{k+1}+\mu^{-1}v_{k+1}\|^2 \\
    &\quad -\frac{A_{k+1}}{2\mu}\|\nabla_{k+1}\|^2 
    +\frac{A_{k+1}}{2\mu}\left( \|\nabla_{k+1}\|^2-\|v_{k+1}\|^2 \right) \\
    &\quad +\frac{A_{k+1}}{\lambda_k(1+\lambda_k\mu)}\left( \frac{1}{2}\|x_{k+1}-y_k+\lambda_k v_{k+1}\|^2-\frac{\sigma_k^2}{2}\|x_{k+1}-y_k\|^2 \right) \\
    \end{aligned}
\end{equation}
We now show that the last expression in the above inequality is positive. Recall that in \Cref{overview} we made the intuitive argument which shows that the ``positive term" of form $\theta_z \|z_k-x_{k+1}+\mu^{-1}v_{k+1}\|^2 + \theta_x \|x_k-x_{k+1}+\mu^{-1}v_{k+1}\|^2$ cannot be small. Formally, the choice of $y_k$ implies that
\begin{equation}
    \notag
    \begin{aligned}
    &\quad \frac{\mu a_{k+1}(1+\mu A_k)}{2(1+\mu A_{k+1})}\|z_k-x_{k+1}+\mu^{-1}v_{k+1}\|^2 + \frac{\mu A_k}{2}\|x_k-x_{k+1}+\mu^{-1}v_{k+1}\|^2 \\
    &\geq \frac{\mu\left( A_{k+1}+\mu(a_{k+1}A_k+A_k A_{k+1}) \right)}{2(1+\mu A_{k+1})} \|y_k-x_{k+1}+\mu^{-1}v_{k+1}\|^2 \\
    &\quad +\frac{\mu A_k a_{k+1}(1+\mu A_k)}{2\left( A_{k+1}+\mu(a_{k+1}A_k+A_k A_{k+1}) \right)}\|x_k-z_k\|^2 \\
    &= \frac{\mu a_{k+1}^2}{2\lambda_k(1+\mu A_{k+1})}\|y_k-x_{k+1}+\mu^{-1}v_{k+1}\|^2+\frac{\mu\lambda_k A_k(1+\mu A_k)}{2a_{k+1}}\|x_k-z_k\|^2
    \end{aligned}
\end{equation}
where we have used the following equation
\begin{equation}
\label{a_k+1}
    a_{k+1}^2=\lambda_k\left( A_{k+1}+\mu(a_{k+1}A_k+A_k A_{k+1}) \right)
\end{equation}
to simplify the expression. We can now deduce from \cref{simpler} that the right hand side of~\cref{long_ineq} is lower bounded by
\begin{equation}
\notag
    \begin{aligned}
    &\frac{\mu a_{k+1}^2}{2\lambda_k(1+\mu A_{k+1})}\|y_k-x_{k+1}+\mu^{-1}v_{k+1}\|^2+\frac{\mu\lambda_k A_k(1+\mu A_k)}{2a_{k+1}}\|x_k-z_k\|^2 \\
    &\quad -\frac{A_{k+1}}{2\mu}\|v_{k+1}\|^2+\frac{A_{k+1}}{\lambda_k}\bigl( \tfrac{1}{2}\|x_{k+1}-y_k+\lambda_k v_{k+1}\|^2-\tfrac{\sigma_k^2}{2}\|x_{k+1}-y_k\|^2 \bigr).
    \end{aligned}
\end{equation}

Now except from the $\|x_k-z_k\|^2$ term which is non-negative, the rest can be written as
\begin{equation}
    \label{quadratic_form}
    \alpha \|x_{k+1}-y_k\|^2+2\beta\left\langle x_{k+1}-y_k,v_{k+1}\right\rangle +\gamma \|v_{k+1}\|^2
\end{equation}
where
\begin{equation}
\notag
    \begin{aligned}
    \alpha &= \frac{\mu a_{k+1}^2}{2\lambda_k(1+\mu A_{k+1})}+\frac{1-\sigma_k^2}{2}\frac{A_{k+1}}{\lambda_k} \\
    \beta &= \frac{a_{k+1}^2}{2\lambda_k(1+\mu A_{k+1})}-\frac{1}{2}A_{k+1} = -\frac{\mu a_{k+1}^2}{2(1+\mu A_{k+1})} \\
    \gamma &= \frac{a_{k+1}^2}{2\mu\lambda_k(1+\mu A_{k+1})}-\frac{A_{k+1}}{2\mu}+\frac{1}{2}\lambda_k A_{k+1} \\
    &= \frac{1}{2}\lambda_k A_{k+1}-\frac{a_{k+1}^2}{2(1+\mu A_{k+1})} = \frac{\mu\lambda_k a_{k+1}^2}{2(1+\mu A_{k+1})}
    \end{aligned}
\end{equation}
where we have used \Cref{a_k+1} to simplify the expressions. Now it's easy to see that the desired inequality holds.
\end{proof}

\noindent We now make some remarks on the previous lemma.
\begin{enumerate}
\item Firstly, we can see from the proof that the choice of $a_{k+1}$ guarantees that the quadratic function \Cref{quadratic_form} is non-negative. The correct way of obtaining $a_{k+1}$ is to first deduce the quadratic function and then determine a proper choice of $a_{k+1}$ such that the function is always non-negative. This approach will be used to derive the update rule of $a_{k+1}$ in the Riemannian setting, where additional parameters need to be introduced due to the distortion phenomenon.

\item Secondly, as we have discussed before, $x_k$ and $z_k$ can both be regarded as an approximate proximal point iterate, and the point $y_k$ is chosen on the segment between $x_k$ and $z_k$ in order to combine these two approaches. 
The ratio $\|x_k-y_k\|:\|y_k-z_k\|$ follows naturally from the analysis and \Cref{simple}, which suggests the correct way of doing this combination.
\end{enumerate}

\Cref{euclidean_convergence} is now a direct corollary of \Cref{potential_dec}.

\begin{theorem}
\label{main_result_euclidean}
(\Cref{euclidean_convergence} restated) For the iterates produced by Algorithm \ref{AHPE_Euclidean}, we have
\begin{equation}
\notag
    \begin{aligned}
    f(x_k)-f(x^*) &\leq \frac{1}{A_k}\left( A_0(f(x_0)-f(x^*))+\frac{1+\mu A_0}{2}\|x_0-x^*\|^2 \right) \\
    &= \mathcal{O}\left( \Pi_{i=1}^{k}\left(1+\max\left\{ \mu\lambda_i,\sqrt{\mu\lambda_i}\right\}\right)^{-1}\right) \\
    \end{aligned}
\end{equation}
\end{theorem}

\begin{proof}
Since $p_0 \geq p_k \geq A_k(f(x_k)-f(x^*))$, we have
\begin{equation}
    \notag
    f(x_k)-f(x^*) \leq \frac{1}{A_k} p_0 = \frac{1}{A_k}\left( A_0(f(x_0)-f(x^*))+\frac{1+\mu A_0}{2}\|x_0-x^*\|^2 \right).
\end{equation}
Note that
\begin{equation}
    \notag
    \begin{aligned}
    a_{k+1} &= A_{k+1} - A_k = \frac{\left(1+2 \mu A_{k}\right) \lambda_{k+1}+\sqrt{\left(1+2 \mu A_{k}\right)^{2} \lambda_{k+1}^{2}+4\left(1+\mu A_{k}\right) A_{k} \lambda_{k+1}}}{2} \\
    &\geq A_k \max\left\{ \mu\lambda_{k+1}, \sqrt{\mu\lambda_{k+1}}\right\},
    \end{aligned}
\end{equation}
so that the conclusion follows.
\end{proof}

\section{Details of \Cref{section_riemann}}
\label{riem_appendix}

\subsection{Some useful properties of \Cref{AHPE_Riemann}}
The following lemma characterize the growth rate of sequence $\{A_k\}$, which is closely related to the convergence rate of \Cref{AHPE_Riemann}.
\begin{lemma}
\label{increasing_A}
For all $k \geq 0$, we have $A_{k+1} = (1+\mu\lambda_k)(\theta_k a_{k+1} + A_k)$.
\end{lemma}

\begin{proof}
  Since
  \begin{equation}
  \notag
      (1-\theta_k)B_k = (1-\theta_k)\theta_k\delta_k B_{k+1} = \frac{\mu}{2}\theta_k\delta_k a_{k+1},
  \end{equation}
  the equation $B_k(1-\theta_k)^2 = \mu\lambda_k\theta_k\left( (1-\theta_k)B_k+\frac{\mu}{2}\delta_k A_k\right)$ can be equivalently written as
  \begin{equation}
      \notag
      \begin{aligned}
      &\qquad (1-\theta_k)\frac{\mu}{2}\theta_k\delta_k a_{k+1} = \mu\lambda_k\theta_k \cdot \frac{\mu}{2}\delta_k(A_k+\theta_k a_{k+1}) \\
      &\Leftrightarrow (1-\theta_k)a_{k+1} = \mu\lambda_k(A_k+\theta_k a_{k+1}) \\
      &\Leftrightarrow A_{k+1}=A_k+a_{k+1} = (1+\mu\lambda_k)(A_k+\theta_k a_{k+1}).
      \end{aligned}
  \end{equation}
  The conclusion follows.
\end{proof}

The next lemma reveals the relationship between the ratio of coefficients $A_k$ and $B_k$ and an important quantity $\xi_k = \frac{a_k}{A_k}$ (defined in the proof of \Cref{simple_convergence_thm}). Recall that in the Euclidean setting, we have the equation $B_k = \frac{1+\mu A_k}{2}$, but the situation is more complex in the Riemannian setting due to the distortion rate $\delta_k$.

\begin{lemma}
For any $k\geq 0$, we have
\begin{equation}
    \notag
    \frac{B_{k+1}}{A_{k+1}} = \frac{1+\mu\lambda_k}{2\lambda_k}\left( \frac{a_{k+1}}{A_{k+1}}\right)^2 = \frac{1+\mu\lambda_k}{2\lambda_k} \xi_{k+1}^2.
\end{equation}
\end{lemma}

\begin{proof}
Recall that we have $A_{k+1}=(1+\mu\lambda_k)(\theta_k a_{k+1}+A_k) = (1+\mu\lambda_k)(A_{k+1}-(1-\theta_k)a_{k+1})$, so that
\begin{equation}
    \notag
    1-\theta_k = \frac{\mu\lambda_k A_{k+1}}{(1+\mu\lambda_k)a_{k+1}}.
\end{equation}
We can then obtain
\begin{equation}
    \notag
    \frac{B_{k+1}}{A_{k+1}} = \frac{\mu}{2}\frac{a_{k+1}}{(1-\theta_k)A_{k+1}} = \frac{1+\mu\lambda_k}{2\lambda_k}\left( \frac{a_{k+1}}{A_{k+1}}\right)^2,
\end{equation}
as desired.
\end{proof}

\subsection{Potential function analysis}
\begin{lemma}[restatement of \Cref{reference_distortion}]
\label{reference_distortion_appendix}
Suppose that $\delta_k>0$ is a valid distortion rate and $B_{k+1} = \frac{B_k}{\theta_k\delta_k}$, then 
\begin{equation}
    \notag
    \begin{aligned}
    B_{k}d_{w_k}^2(z_k,x^*)-B_{k+1}d_{w_{k+1}}^2(z_{k+1},x^*)
    &\geq (1-\theta_k)B_{k+1}\left(\frac{2}{\mu}(f(w_{k+1})-f(x^*))-\frac{1}{\mu^2}\|\nabla_{k+1}\|^2\right) \\
    &\quad + \theta_k(1-\theta_k)B_{k+1}\left\| \mathtt{Exp}_{w_{k+1}}^{-1}(z_k)-\mathtt{Exp}_{w_{k+1}}^{-1}(x_{k+1})+\mu^{-1}v_{k+1} \right\|^2
    \end{aligned}
\end{equation}
\end{lemma}

\begin{proof}

Since $\delta_k$ is a valid distortion rate, we have 
\begin{equation}
\label{apply_distortion_ineq}
    B_{k}d_{w_k}^2(z_k,x^*) \geq \frac{B_{k}}{\delta_k}d_{w_{k+1}}^2(z_k,x^*)
\end{equation}
This implies that
\begin{subequations}\label{Riemann_distance_ineq}
\begin{align}
    &\quad B_{k+1}d_{w_{k+1}}^2(z_{k+1},x^*)-B_{k}d_{w_k}^2(z_k,x^*)
    \leq B_{k+1}d_{w_{k+1}}^2(z_{k+1},x^*)-\theta_k B_{k+1} d_{w_{k+1}}^2(z_k,x^*)\label{Riemann_distance_ineq1} \\
    &= (1-\theta_k)B_{k+1}\left( \frac{1}{1-\theta_k}d_{w_{k+1}}(z_{k+1},x^*) -\frac{\theta_k}{1-\theta_k}d_{w_{k+1}}(z_k,x^*) \right)^2 \label{Riemann_distance_ineq2} \\
    &\quad -\frac{\theta_k}{1-\theta_k}\left( d_{w_{k+1}}(z_{k+1},x^*)-d_{w_{k+1}}(z_{k},x^*)\right)^2 \nonumber\\
    &= (1-\theta_k)B_{k+1}\left\| \mathtt{Exp}_{w_{k+1}}^{-1}(x^*)-\mathtt{Exp}_{w_{k+1}}^{-1}(x_{k+1})+\mu^{-1}v_{k+1} \right\|^2 \label{Riemann_distance_ineq3}\\
    &\quad -\theta_k(1-\theta_k)B_{k+1}\left\| \mathtt{Exp}_{w_{k+1}}^{-1}(z_k)-\mathtt{Exp}_{w_{k+1}}^{-1}(x_{k+1})+\mu^{-1}v_{k+1} \right\|^2 \nonumber
\end{align}
\end{subequations}
where  \Cref{Riemann_distance_ineq1} follows from \Cref{apply_distortion_ineq} and $\theta_k B_{k+1}=\frac{B_k}{\delta_k}$, \Cref{Riemann_distance_ineq2} uses \Cref{simple}, and \Cref{Riemann_distance_ineq3} follows from the definition of $z_{k+1}$. On the other hand, by strong convexity of $f$, we have
\begin{equation}
    \notag
    \begin{aligned}
    f(x^*)-f(w_{k+1}) &\geq \left\langle \mathtt{Exp}_{w_{k+1}}^{-1}(x^*),\nabla_{k+1}\right\rangle +\frac{\mu}{2}\|\mathtt{Exp}_{w_{k+1}}^{-1}\|^2 \\
    &= \frac{\mu}{2}\|\mathtt{Exp}_{w_{k+1}}^{-1}(x^*)+\mu^{-1}\nabla_{k+1}\|^2 - \frac{1}{2\mu}\|\nabla_{k+1}\|^2 \\
    &= \frac{\mu}{2}\left\| \mathtt{Exp}_{w_{k+1}}^{-1}(x^*)-\mathtt{Exp}_{w_{k+1}}^{-1}(x_{k+1})+\mu^{-1}v_{k+1} \right\|^2- \frac{1}{2\mu}\|\nabla_{k+1}\|^2
    \end{aligned}
\end{equation}
The conclusion follows by plugging this inequality into \Cref{Riemann_distance_ineq}. Note that the steps after \Cref{Riemann_distance_ineq1} are essentially the same as the Euclidean setting, because all the calculations are done in the tangent space $T_{w_{k+1}}\mathcal{M}$.
\end{proof}

We then proceed to derive a Riemannian analog of \Cref{potential_dec}, where we proved the potential decrease in the Euclidean setting. By following the same approach as \Cref{potential_dec}, we can see that the inequality would involve an additional point $y_k'$. 
\begin{lemma}[restatement of \Cref{update_distortion}]
\label{update_distortion_appendix}
Suppose that $a_{k+1}= A_{k+1}-A_k= \frac{2}{\mu}(1-\theta_k)B_{k+1}$, then 
\begin{equation}
\label{update_distortion_eq_appendix}
    \begin{aligned}
    p_k-p_{k+1} &\geq \frac{\mu}{2}(\theta_k a_{k+1}+A_k)\left\| \mathtt{Exp}_{w_{k+1}}^{-1}(y_k')-\mathtt{Exp}_{w_{k+1}}^{-1}(x_{k+1})+\mu^{-1}v_{k+1} \right\|^2 \\
    &+\frac{A_{k+1}}{2\lambda_k\sigma_k}\left\| \mathtt{Exp}_{w_{k+1}}^{-1}(y_k)-\mathtt{Exp}_{w_{k+1}}^{-1}(x_{k+1})-\lambda_k v_{k+1} \right\|^2 \\
    &+ \frac{\mu\theta_{k}a_{k+1}A_k}{2(A_k+\theta_{k}a_{k+1})}d_{w_{k+1}}^2(x_k,z_k) - \frac{\sigma_k A_{k+1}}{2\lambda_k}d_{w_{k+1}}^2(x_{k+1},y_k)-\frac{A_{k+1}}{2\mu}\|v_{k+1}\|^2
    \end{aligned}
\end{equation}
where 
\begin{equation}
    \label{def_yk_prime_appendix}
    y_k' = \mathtt{Exp}_{w_{k+1}}\left( \frac{A_k}{A_k+\theta_{k}a_{k+1}}\mathtt{Exp}_{w_{k+1}}^{-1}(x_k)+\frac{\theta_{k}a_{k+1}}{A_k+\theta_{k}a_{k+1}}\mathtt{Exp}_{w_{k+1}}^{-1}(z_k) \right)
\end{equation}
\end{lemma}

\begin{proof}
Recall the the argument in \Cref{potential_dec} basically uses strong convexity and the definition of Euclidean \textit{iprox} to lower bound the potential decrease with a quadratic function, and the choice of parameters ensure that the quadratic is positive definite. In the Riemannian setting, since `vector' on a manifold is undefined, we need to work with vectors in a tangent space instead. In the following, we work in the tangent space $\mathcal{T}_{w_{k+1}}$. This choice is quite natural, since straightforwardly generalizing of the proof of \Cref{potential_dec} would involve exponential maps at $w_{k+1}$. In $\mathcal{T}_{w_{k+1}}$, our goal is to derive a quadratic function to lower bound $p_k-p_{k+1}$.

Strong convexity implies that
\begin{equation}
\notag
    f(x_k) \geq f(w_{k+1})+\frac{\mu}{2}\|\mathtt{Exp}_{w_{k+1}}^{-1}(x_k)+\mu^{-1}\nabla_{k+1}\|^2-\frac{1}{2\mu}\|\nabla_{k+1}\|^2
\end{equation}
and 
\begin{equation}
\notag
    f(x_{k+1}) \geq f(w_{k+1})+\frac{\mu}{2}\|\mu^{-1}v_{k+1}\|^2-\frac{1}{2\mu}\|\nabla_{k+1}\|^2,
\end{equation}
and the definition of Riemannian \textit{iprox} operator \Cref{Riemann_iprox} implies that
\begin{equation}
    \notag
    \begin{aligned}
    & \frac{\sigma_k^2}{2}\|\mathtt{Exp}_{w_{k+1}}^{-1}(x_{k+1})-\mathtt{Exp}_{w_{k+1}}^{-1}(y_k)\|^2 \geq \frac{1}{2}\|\mathtt{Exp}_{w_{k+1}}^{-1}(x_{k+1})-\mathtt{Exp}_{w_{k+1}}^{-1}(y_k)+\lambda_k v_{k+1}\|^2 \\
    &\quad +\lambda_k(1+\lambda_k\mu)\left( f(x_{k+1})-f(w_{k+1})+\frac{1}{2\mu}\left( \|\nabla_{k+1}\|^2-\|v_{k+1}\|^2\right)\right)
    \end{aligned}
\end{equation}
Combining the above inequalities, we have
\begin{subequations}\label{Riemann_combine}
    \begin{align}
    p_k-p_{k+1} 
    &= \left( B_k d_{w_k}^2(z_k,x^*)-B_{k+1} d_{w_{k+1}}^2(z_{k+1},x^*)+\frac{2}{\mu}(1-\theta_k)B_{k+1}(f(x^*)-f(w_{k+1})) \right) \label{Riemann_combine1}\\
    &\quad + A_k(f(x_k)-f(w_{k+1}))+A_{k+1}(f(w_{k+1})-f(x_{k+1})) \nonumber\\
    &\geq \frac{\mu A_k}{2}\left\| \mathtt{Exp}_{w_{k+1}}^{-1}(x_k)-\mathtt{Exp}_{w_{k+1}}^{-1}(x_{k+1})+\mu^{-1}v_{k+1}\right\|^2  \nonumber \\
    &\quad +\frac{\mu}{2}\theta_k a_{k+1}\left\| \mathtt{Exp}_{w_{k+1}}^{-1}(z_k)-\mathtt{Exp}_{w_{k+1}}^{-1}(x_{k+1})+\mu^{-1}v_{k+1}\right\|^2 \nonumber\\
    &\quad +\frac{A_{k+1}}{2\lambda_k\sigma_k}\left\| \mathtt{Exp}_{w_{k+1}}^{-1}(y_k)-\mathtt{Exp}_{w_{k+1}}^{-1}(x_{k+1})-\lambda_k v_{k+1} \right\|^2  \nonumber\\
    &\quad - \frac{\sigma_k A_{k+1}}{2\lambda_k}d_{w_{k+1}}^2(x_{k+1},y_k)-\frac{A_{k+1}}{2\mu}\|v_{k+1}\|^2 \nonumber
    \end{align}
\end{subequations}
where we use the condition $a_{k+1} = \frac{2}{\mu}(1-\theta_k)B_{k+1}$ in \eqref{Riemann_combine1}. Finally, \Cref{simple}  implies that
\begin{equation}
\notag
    \begin{aligned}
    &\quad \frac{\mu A_k}{2}\left\| \mathtt{Exp}_{w_{k+1}}^{-1}(x_k)-\mathtt{Exp}_{w_{k+1}}^{-1}(x_{k+1})+\mu^{-1}v_{k+1}\right\|^2 \\
    &\qquad +\frac{\mu}{2}\theta_k a_{k+1}\left\| \mathtt{Exp}_{w_{k+1}}^{-1}(z_k)-\mathtt{Exp}_{w_{k+1}}^{-1}(x_{k+1})+\mu^{-1}v_{k+1}\right\|^2 \\
    &= \frac{\mu}{2}\left( \theta_k a_{k+1}+A_k \right) \left\| \mathtt{Exp}_{w_{k+1}}^{-1}(y_k')-\mathtt{Exp}_{w_{k+1}}^{-1}(x_{k+1})+\mu^{-1}v_{k+1}\right\|^2 \\
    &\qquad + \frac{\mu\theta_{k}a_{k+1}A_k}{2(A_k+\theta_{k}a_{k+1})}d_{w_{k+1}}^2(x_k,z_k).
    \end{aligned}
\end{equation}
The conclusion follows. The final equation in the proof explains why $y_k'$ would appear in \Cref{update_distortion_eq_appendix}.
\end{proof}

\subsection{Convergence without the additional distortion}

In the Riemannian setting, it is not \textit{guaranteed} that $y_k$ is the same as $y_k'$, and this may give rise to the \textit{additional distortion}, as shown in \Cref{update_distortion}. However, recall that our definition of \textit{iprox} allows flexible choices of $x_{k+1}, w_{k+1}$ and $v_{k+1}$. We can see that in some special cases, we still have $y_k=y_k'$. The following proposition provides sufficient condition for this to hold. It can be easily derived from the definition of $y_k$ and $y_k'$.

\begin{proposition}
\label{no_addis_prop}
Suppose that $w_{k+1}$ lies on the geodesic connecting $x_k$ and $z_k$, then $y_k = y_k'$.
\end{proposition}

We now move on to theoretical analysis under the condition $y_k=y_k'$. The right hand side of \eqref{update_distortion_eq} is a quadratic function of $\mathtt{Exp}_{w_{k+1}}^{-1}(y_k)-\mathtt{Exp}_{w_{k+1}}^{-1}(x_{k+1})$ and $v_{k+1}$, after ignoring the non-negative $d_{w_{k+1}}(x_k,z_k)$ term. We can then prove the following lemma for potential decrease. The equation $A_{k+1} = (1+\mu\lambda_k)(\theta_k a_{k+1} + A_k)$ plays a crucial role in the proof.

\begin{lemma}
\label{rescale}
Suppose that $\sigma_k < 1$, then
\begin{equation}
\notag
    \begin{aligned}
    \frac{(1-\sigma_k)A_{k+1}}{2\lambda_k}d_{w_{k+1}}^2(x_{k+1},y_k)
    &\leq \frac{\mu}{2}(\theta_k a_{k+1}+A_k)\left\| \mathtt{Exp}_{w_{k+1}}^{-1}(y_k)-\mathtt{Exp}_{w_{k+1}}^{-1}(x_{k+1})+\mu^{-1}v_{k+1} \right\|^2 \\
    &\quad+\frac{A_{k+1}}{2\lambda_k}\left\| \mathtt{Exp}_{w_{k+1}}^{-1}(y_k)-\mathtt{Exp}_{w_{k+1}}^{-1}(x_{k+1})-\lambda_k v_{k+1} \right\|^2 \\
    &\quad- \frac{\sigma_k A_{k+1}}{2\lambda_k}d_{w_{k+1}}^2(x_{k+1},y_k) - \frac{A_{k+1}}{2\mu}\|v_{k+1}\|^2
    \end{aligned}
\end{equation}
\end{lemma}

\begin{proof}

First note that the difference of the right hand side and left hand side of the inequality can be written in the following form (where we omit the $d_{w_{k+1}}^2(x_k,z_k)$ term, which is non-negative):
\begin{equation}
    \notag
    \mathtt{RHS}-\mathtt{LHS} = \alpha d_{w_{k+1}}^2(x_{k+1},y_k) + 2\beta \left\langle \mathtt{Exp}_{w_{k+1}}^{-1}(y_k)-\mathtt{Exp}_{w_{k+1}}^{-1}(x_{k+1}), v_{k+1} \right\rangle + \gamma \|v_{k+1}\|^2
\end{equation}
where
\begin{equation}
\notag
    \begin{aligned}
    \alpha &= \frac{\mu}{2}(\theta_k a_{k+1}+A_k)+\frac{(1-\sigma_k)A_{k+1}}{2\lambda_k} = \frac{A_{k+1}}{2}\left( \frac{\mu}{1+\mu\lambda_k}+\frac{1-\sigma_k}{\lambda_k}\right) \\
    \beta &= \frac{1}{2}(\theta_k a_{k+1}+A_k)-\frac{A_{k+1}}{2} = - \frac{A_{k+1}}{2}\cdot\frac{\mu\lambda_k}{1+\mu\lambda_k} \\
    \gamma &= \frac{1}{2\mu}(\theta_k a_{k+1}+A_k)+\frac{\lambda_k A_{k+1}}{2}-\frac{A_{k+1}}{2\mu} = \frac{A_{k+1}}{2}\cdot\frac{\mu\lambda_k^2}{1+\mu\lambda_k}.
    \end{aligned}
\end{equation}
Note that
\begin{equation}
    \notag
    \beta^2 = \left( \alpha - \frac{(1-\sigma_k)A_{k+1}}{2\lambda_k}\right) \gamma,
\end{equation}
we can thus obtain
\begin{equation}
    \notag
    \mathtt{RHS}-\mathtt{LHS} \geq \frac{(1-\sigma_k)A_{k+1}}{2\lambda_k} d_{w_{k+1}}^2(x_{k+1},y_k)
\end{equation}
as desired.
\end{proof}

Combining \Cref{reference_distortion} and \Cref{rescale}, we can see that the potential sequence $\{p_k\}$ is non-increasing:
\begin{corollary}
\label{cor_simple}
 Suppose that $y_k=y_k'$, then the following inequality holds:
 \begin{equation}
     \notag
     p_k-p_{k+1} \geq \frac{(1-\sigma_k)A_{k+1}}{2\lambda_k} d_{w_{k+1}}^2(x_{k+1},y_k) + \frac{\mu\theta_{k}a_{k+1}A_k}{2(A_k+\theta_{k}a_{k+1})}d_{w_{k+1}}^2(x_k,z_k).
\end{equation}
In particular, we have $p_{k+1} \leq p_k$, so that $p_k \leq p_0$ for all $k \geq 1$.
\end{corollary}

Finally, we can prove the following theorem, which says that if $w_{k+1}$ is chosen on the geodesic connecting $x_k$ and $z_k$, then \Cref{AHPE_Riemann} provably achieves eventual acceleration with arbitrary initialization.

\begin{theorem}[restatement of \Cref{simple_convergence_thm}]
\label{simple_convergence_thm_appendix}
Suppose that in \Cref{AHPE_Riemann}, we choose $\lambda_k = \lambda$ and $w_{k+1}$ lies on the geodesic connecting $x_k$ and $z_k$ s.t. $d(w_{k+1},y_k) = \mathcal{O}(1)$, then we have
\begin{equation}
    \label{convergence_ineq_appendix}
    \begin{aligned}
    f(x_k)-f(x^*) &\leq \frac{p_0}{A_k},\quad
    d_{w_{k+1}}^2(z_k,x^*) &\leq \frac{p_0}{B_k} \leq \frac{2 p_0}{\mu a_k}.
    \end{aligned}
\end{equation}
Moreover, we have
\begin{equation}
    \notag
    \lim_{k\to +\infty}\frac{A_{k+1}}{A_k} = 1+\mu\lambda + \sqrt{\mu\lambda(1+\mu\lambda)}
\end{equation}
\end{theorem}

\begin{proof}
The first two inequalities follow from  \Cref{cor_simple} and
\begin{equation}
    \label{eventual_acc}
    a_{k+1} = (1-\theta_k) \frac{B_k}{\delta_k\theta_k} = 2\mu^{-1}(1-\theta_k) B_{k+1} < 2\mu^{-1} B_{k+1}, \quad \forall k \geq 0.
\end{equation}
We now prove \eqref{eventual_acc}. This is equivalent to 
\begin{equation}
    \notag
    \lim_{k\to +\infty} \frac{a_k}{A_k} = \sqrt{\frac{\mu\lambda}{1+\mu\lambda}}.
\end{equation}
Define $\xi_k = \frac{a_k}{A_k}$ for $k\geq 1$. Note that the update of $\theta_{k}$ and $a_{k+1}$ in Algorithm \ref{AHPE_Riemann} implies that
\begin{equation}
    \notag
    \delta_k a_{k+1}^2 - 2\lambda \left( \frac{\mu}{2}\delta_k A_k+B_k \right) a_{k+1}-2\lambda A_k B_k = 0
\end{equation}
Thus
\begin{equation}
    \begin{aligned}
    \delta_k a_{k+1}^2 &= 2\lambda \left( B_k A_{k+1}+\frac{\mu}{2}\delta_k A_k a_{k+1} \right) \\
    (1+\mu\lambda) a_{k+1}^2 &= 2\delta_k^{-1}\lambda A_{k+1}\left( B_k+\frac{\mu}{2}\delta_k a_{k+1} \right) = 2\lambda A_{k+1} B_{k+1}
    \end{aligned}
\end{equation}
As a result, we have $\frac{B_k}{A_k} = \frac{1+\mu\lambda}{2\lambda}\xi_k^2$. The above derivations only holds for $k\geq 1$, we artificially define $\xi_0 = \sqrt{\frac{2\lambda}{1+\mu\lambda}\frac{B_k}{A_k}}$, so that for all $k\geq 0$, rewrite the equation $B_{k+1} = \frac{B_k}{\delta_k} + \frac{\mu}{2}a_{k+1}$ in terms of $\xi$ as
\begin{equation}
    \notag
    \delta_k\frac{1+\mu\lambda}{2\lambda}\xi_{k+1}^2 = \frac{1+\mu\lambda}{2\lambda}\xi_k^2(1-\xi_{k+1})+\frac{\mu}{2}\delta_{k+1}\xi_{k+1}
\end{equation}
or equivalently,
\begin{equation}
    \label{recursive_xi}
    \delta_k\xi_{k+1}^2 = \xi_k^2(1-\xi_{k+1}) + \frac{\mu\lambda}{1+\mu\lambda}\delta_k\xi_{k+1}
\end{equation}
Before proceeding to analyze the recursive equation \eqref{recursive_xi}, we first prove that $\lim_{k\to+\infty}\delta_k = 1$. This is in fact necessary since otherwise $\{\xi_k\}$ would not converge to the fixed point $\sqrt{\frac{\mu\lambda}{1+\mu\lambda}}$.

Since $\lim_{k\to +\infty} A_k = +\infty$, we have $x_k \to x^*$ and $d_{w_{k+1}}(x_{k+1},y_k) \to 0$, by \Cref{cor_simple}. By assumption, $d(w_{k+1},y_k)$ is bounded, so that
\begin{equation}
    \notag
    d(x_{k+1},y_k) \leq d_{w_{k+1}}(x_{k+1},y_k) + 2d(w_{k+1},y_k)
\end{equation}
is bounded, which implies that the sequence $\{y_k\}$ is bounded. Thus $\{w_k\}$ is also bounded.

Since $A_{k+1} \geq (1+2\mu\lambda) A_k$, we have $a_{k+1} = A_{k+1}-A_k \geq 2\mu\lambda A_k$, so that $\lim_{k\to +\infty} a_k = +\infty$ and $d_{w_{k+1}}^2(z_k,x^*) \leq \frac{p_0}{a_k} \to 0$. Since $w_{k+1} = \mathcal{O}(1)$, the distortion inequality \Cref{ahn_distortion} implies that $d(z_k,x^*) \to 0$. Note that $w_{k+1}$ lies on the geodesic connecting $x_k$ and $z_k$, and $\mathcal{M}$ has non-positive curvature, we have
\begin{equation}
    \notag
    d(w_{k+1},x^*) \leq \max\{ d(x_k,x^*), d(z_k,x^*) \} \to 0.
\end{equation}
Hence $\delta_k = T_{K}(d(w_k,z_k)) \to 1$ as $k \to +\infty$.

We now return to \eqref{recursive_xi}. We first show that for any $\varepsilon>0$, we have
\begin{equation}
    \notag
    \liminf_{k\to +\infty} \xi_k \geq (1-\varepsilon) \sqrt{\frac{\mu\lambda}{1+\mu\lambda}}
\end{equation}
Since $d(w_{k+1},y_k) \leq D_k \to 0$ by assumption, and $y_k \to x^*$, we have $w_{k+1} \to x^*$. The definition of $\delta_k$ then implies that $\lim_{k\to +\infty} \delta_k = 1$.

The recursive relation \eqref{recursive_xi} can be rewritten as
\begin{equation}
    \notag
    \delta_k\xi_{k+1}\left( \xi_{k+1}-\frac{\mu\lambda}{1+\mu\lambda} \right) = \xi_k^2(1-\xi_{k+1})
\end{equation}
Note that: if $\delta_k$ becomes larger  and $\xi_k$ becomes smaller, then $\xi_{k+1}$ also becomes smaller. Based on this observation, we first choose $k_0$ such that $\delta_k \leq 1+\varepsilon \sqrt{\frac{\mu\lambda}{1+\mu\lambda}}$ for all $k \geq k_0$, and then construct a \textit{reference sequence} $\{\zeta_k\}_{k\geq k_0}$ defined as
\begin{equation}
\notag
    \zeta_{k_0} = \xi_{k_0},\quad \delta\zeta_{k+1}\left( \zeta_{k+1}-\frac{\mu\lambda}{1+\mu\lambda} \right) = \zeta_k^2(1-\zeta_{k+1}),\quad \delta = 1+\varepsilon \sqrt{\frac{\mu\lambda}{1+\mu\lambda}}
\end{equation}
Then we have $\xi_k \geq \zeta_k$ for all $k \geq k_0$. Alternatively, we can write the recursion above as $\zeta_{k+1} = \varphi(\zeta_k)$, where
\begin{equation}
    \label{fix_point}
    \varphi(x) = \frac{1}{2\delta}\left( \frac{\mu\lambda}{1+\mu\lambda}\delta-x^2+\sqrt{\left( x^2 -\frac{\mu\lambda}{1+\mu\lambda}\delta\right)^2+4\delta x^2} \right)
\end{equation}
We have
\begin{equation}
    \notag
    \varphi'(x) = -\frac{x}{\delta}+\frac{x\left( x^2 -\frac{\mu\lambda}{1+\mu\lambda}\delta\right)+2\delta x}{\delta\sqrt{\left( x^2 -\frac{\mu\lambda}{1+\mu\lambda}\delta\right)^2+4\delta x^2}}
\end{equation}
The observation made above implies that $\varphi'(x) \geq 0$. On the other hand,
\begin{equation}
    \begin{aligned}
    \varphi'(x) < 1 &\Leftrightarrow \left( x\left( x^2 -\frac{\mu\lambda}{1+\mu\lambda}\delta\right)+2\delta x \right)^2 < (x+\delta)^2 \left( \left( x^2 -\frac{\mu\lambda}{1+\mu\lambda}\delta\right)^2+4\delta x^2 \right) \\
    &\Leftarrow 4\delta x^2\left( x^2 -\frac{\mu\lambda}{1+\mu\lambda}\delta\right)^2 + 4\delta^2 x^2 < (x+\delta)^2 \cdot 4\delta x^2
    \end{aligned}
\end{equation}
which trivially holds, since $\delta > 1$. Since $\varphi$ is continuously differentiable, we have $\sup_{x \in [0,1]} \varphi'(x) < 1$ i.e. $\varphi$ is a contraction mapping. Since $\zeta_k \in [0,1]$, $\forall k\geq k_0$, it converges exponentially fast to a fixed point of $\varphi$, which is the positive root of the equation $x^2+(\delta-1)x-\frac{\mu\lambda}{1+\mu\lambda}\delta=0$. It's easy to check that this root is larger than $(1-\varepsilon) \sqrt{\frac{\mu\lambda}{1+\mu\lambda}}$, so that
\begin{equation}
\notag
    \liminf_{k\to +\infty} \xi_k \geq \liminf_{k\to +\infty} \zeta_k \geq (1-\varepsilon) \sqrt{\frac{\mu\lambda}{1+\mu\lambda}}
\end{equation}
To prove the desired result, it remains show that
\begin{equation}
    \notag
    \limsup_{k\to +\infty} \xi_k \leq \sqrt{\frac{\mu\lambda}{1+\mu\lambda}}
\end{equation}
This can be similarly shown by constructing a reference sequence with $\delta = 1$ in the recursion, and the reference sequence converges to the fixed point corresponding to $\delta=1$, which is $\sqrt{\frac{\mu\lambda}{1+\mu\lambda}}$.
\end{proof}

Although \Cref{simple_convergence_thm} shows that \Cref{AHPE_Riemann} eventually achieves acceleration in the sense of \Cref{def_eventual_acc}, it might also be helpful to know how fast the sequence $\{\tau_k\}$ (cf. \Cref{def_eventual_acc}) achieves the order of $\mathcal{O}\left(\sqrt{\frac{\mu}{L}}\right)$ i.e. how long the `burn-in' period takes to achieve full acceleration. Since at this point we are focusing on acceleration for smooth strongly-convex functions, in the following we always assume that $f$ is $L$-smooth.

\begin{lemma}
Suppose that $d(w_{k+1},y_k) = \mathcal{O}(1)$, and $\lambda_k=\lambda = \frac{c}{L}$, where $c \in (0,1)$ is a numerical constant, then
\begin{equation}
    \notag
    \delta_k - 1 \leq C_0 \left( 1+c \frac{\mu}{L}\right)^{-k}
\end{equation}
where $C_0$ is a constant that may depend on $L,\mu$ and initialization, but independent of $k$.
\end{lemma}

\begin{proof}
Let $D$ be a uniform upper bound of $d(w_{k+1},y_k)$. Since $a_k \geq c \frac{\mu}{L}A_k \geq c \frac{\mu}{L} \left( 1+c \frac{\mu}{L}\right)^k A_0$, we have
\begin{equation}
\notag
    d_{w_{k+1}}^2(z_k,x^*) \leq \frac{2L}{\mu^2 A_0} \left( 1+c \frac{\mu}{L}\right)^{-k}p_0
\end{equation}
Recall that in the proof of \Cref{simple_convergence_thm} we have shown that $\{w_k\}$ is bounded, and it's easy to see that the upper bound only depends on initialization and $d(w_{k+1},y_k)$, by \Cref{ahn_distortion} we have
\begin{equation}
    \notag
    d^2(z_k,x^*) \leq \frac{2C_1 L}{c \mu^2 A_0} \left( 1+c \frac{\mu}{L}\right)^{-k}p_0
\end{equation}
for some $C_1 \geq 1$ that only depends on initialization and $D$. Since $w_{k+1}$ lies on the geodesic between $x_k$ and $z_k$, we have
\begin{equation}
    \notag
    \begin{aligned}
    d^2(w_{k+1},x^*) &\leq \max\left\{ d^2(x_k,x^*), d^2(z_k,x^*) \right\} \\
    &\leq \max\left\{ 2\mu^{-1}(f(x_k)-f(x^*)), d^2(z_k,x^*) \right\} \leq  \frac{2C_1 L}{c\mu^2 A_0} \left( 1+c \frac{\mu}{L}\right)^{-k}p_0
    \end{aligned}
\end{equation}
As a result,
\begin{equation}
    d^2(w_k,z_k) \leq 2\left( d^2(w_k,x^*)+d^2(z_k,x^*)\right) \leq \frac{12 C_1 L}{c \mu^2 A_0} \left( 1+c \frac{\mu}{L}\right)^{-(k-1)}p_0
\end{equation}
Finally since $T_{K}(r) = 1 +\mathcal{O}(r^2)$ for small $r$, we have
\begin{equation}
    \notag
    \delta_k - 1 =\mathcal{O}\left( \left( 1+c \frac{\mu}{L}\right)^{-k} \right),
\end{equation}
as desired.
\end{proof}

\begin{theorem}[restatement of \Cref{achieve_acc_time}]
Suppose that $d(w_{k+1},y_k) = \mathcal{O}(1)$ and $\lambda_k=\lambda=\mathcal{O}\left(\frac{1}{L}\right)$, then we have $\xi_k \geq \frac{1}{2}\sqrt{\frac{\mu\lambda}{1+\mu\lambda}}$ after $T = \widetilde{\mathcal{O}}\left( \frac{L}{\mu}\right)$ iterations, where $\widetilde{\mathcal{O}}$ hides logarithmic terms which may depend on $L,\mu$ and the initialization. As a result, \Cref{AHPE_Riemann} achieves acceleration in at most $\widetilde{\mathcal{O}}\left( \frac{L}{\mu}\right)$ iterations.
\end{theorem}

\begin{proof}
  We consider the recursive equation of $\xi_k$ derived in the proof of \Cref{simple_convergence_thm}:
  \begin{equation}
      \label{rec_xi}
      \delta_k \xi_{k+1}\left(\xi_{k+1}-\frac{\mu\lambda}{1+\mu\lambda}\right) = \xi_k^2(1-\xi_{k+1})
  \end{equation}
  The previous lemma implies that 
  \begin{equation}
      \label{small_delta}
      \delta_k \leq 1+\sqrt{\frac{\mu\lambda}{1+\mu\lambda}}
  \end{equation}
  holds after $\widetilde{\mathcal{O}}\left( \frac{L}{\mu}\right)$ iterations, where $\widetilde{\mathcal{O}}$ hides logarithmic terms. In the following, we study how many iterations are needed for $\xi_{k} \geq \frac{1}{2}\sqrt{\frac{\mu\lambda}{1+\mu\lambda}}$ after \eqref{small_delta} is guaranteed to hold.
  
  Indeed, note that smaller $\delta_k$ and larger $\xi_k$ implies a larger $\xi_{k+1}$ in \eqref{rec_xi}, it suffices to consider the case $\delta_k = \delta = 1+\sqrt{\frac{\mu\lambda}{1+\mu\lambda}}$.
  
  Now we study the behavior of $\varphi(x)$ defined in \eqref{fix_point} more carefully. Its derivative $\varphi'(x)$ can be written as
  \begin{equation}
      \notag
      \varphi'(x) = \frac{\frac{2+\mu\lambda}{1+\mu\lambda}\delta x}{\left(x^2+\frac{2+\mu\lambda}{1+\mu\lambda}\delta\right)\sqrt{x^4+\frac{4+2\mu\lambda}{1+\mu\lambda}\delta x^2} +\left( x^4+\frac{4+2\mu\lambda}{1+\mu\lambda}\delta x^2 \right) }
  \end{equation}
  Hence for all $\delta, x > 0$ we have $\varphi'(x) \leq \frac{1}{\sqrt{2}}$. This implies that with a constant $\delta$, \eqref{fix_point} converges to its fixed point in $\widetilde{\mathcal{O}}(1)$ iterations. Since for $\delta = 1+\sqrt{\frac{\mu\lambda}{1+\mu\lambda}}$, its fixed point is larger than $1+\frac{1}{2}\sqrt{\frac{\mu\lambda}{1+\mu\lambda}}$, we conclude that a total number of $\widetilde{\mathcal{O}}\left( \frac{L}{\mu}\right)$ iterations are needed for $\xi_k \geq \frac{1}{2}\sqrt{\frac{\mu\lambda}{1+\mu\lambda}}$ to hold.
\end{proof}

\subsection{The general case}
This subsection provides details and proofs of our main results for the general case, where the additional distortion is present. We begin with the following result, which shows that we need to control the distance between $y_k$ and $y_k'$.
\begin{lemma}
\label{rescale2}
Suppose that $\sigma_k < 1$, then
\begin{equation}
    \label{rescale2_ineq}
    \begin{aligned}
    p_k-p_{k+1} &\geq \frac{(1-\sigma_k)A_{k+1}}{4\lambda_k} d_{w_{k+1}}^2(x_{k+1},y_k) + \frac{\mu\theta_{k}a_{k+1}A_k}{2(A_k+\theta_{k}a_{k+1})}d_{w_{k+1}}^2(x_k,z_k) \\
    &\quad +\frac{(1-\sigma_k)A_{k+1}}{6} \sqrt{\frac{\mu\lambda_k}{1+\mu\lambda_k}}d_{w_{k+1}}(x_{k+1},y_k) \|v_{k+1}\| \\
    &\quad - \mu(\theta_k a_{k+1}+A_k) d_{w_{k+1}}(y_k,y_k') \cdot \left\| \mathtt{Exp}_{w_{k+1}}^{-1}(y_k)-\mathtt{Exp}_{w_{k+1}}^{-1}(x_{k+1})+\mu^{-1}v_{k+1} \right\|
    \end{aligned}
\end{equation}
\end{lemma}

\begin{proof}
Note that
\begin{equation}
    \notag
    \begin{aligned}
    &\quad \left\| \mathtt{Exp}_{w_{k+1}}^{-1}(y_k')-\mathtt{Exp}_{w_{k+1}}^{-1}(x_{k+1})+\mu^{-1}v_{k+1} \right\|^2 \\
    &\geq \left\| \mathtt{Exp}_{w_{k+1}}^{-1}(y_k)-\mathtt{Exp}_{w_{k+1}}^{-1}(x_{k+1})+\mu^{-1}v_{k+1} \right\|^2 \\
    &\quad+ 2\left\langle  \mathtt{Exp}_{w_{k+1}}^{-1}(y_k') - \mathtt{Exp}_{w_{k+1}}^{-1}(y_k),  \mathtt{Exp}_{w_{k+1}}^{-1}(y_k)-\mathtt{Exp}_{w_{k+1}}^{-1}(x_{k+1})+\mu^{-1}v_{k+1} \right\rangle \\
    &\geq \left\| \mathtt{Exp}_{w_{k+1}}^{-1}(y_k)-\mathtt{Exp}_{w_{k+1}}^{-1}(x_{k+1})+\mu^{-1}v_{k+1} \right\|^2 \\
    &\quad - 2 d_{w_{k+1}}(y_k,y_k') \cdot \left\| \mathtt{Exp}_{w_{k+1}}^{-1}(y_k)-\mathtt{Exp}_{w_{k+1}}^{-1}(x_{k+1})+\mu^{-1}v_{k+1} \right\|
    \end{aligned}
\end{equation}
where the last line follows from Cauchy-Schwarz inequality. The remaining steps of the proof is similar to \Cref{rescale}, except that we also need to incorporate the $\left\langle \mathtt{Exp}_{w_{k+1}}^{-1}(y_k)-\mathtt{Exp}_{w_{k+1}}^{-1}(x_{k+1}), v_{k+1}\right\rangle$ into the bound. Indeed we have
\begin{equation}
    \notag
    \begin{aligned}
    &\quad \alpha d_{w_{k+1}}^2(x_{k+1},y_k) + 2\beta \left\langle \mathtt{Exp}_{w_{k+1}}^{-1}(y_k)-\mathtt{Exp}_{w_{k+1}}^{-1}(x_{k+1}), v_{k+1} \right\rangle + \gamma \|v_{k+1}\|^2 \\
    &\geq \frac{(1-\sigma_k)A_{k+1}}{4\lambda_k} d_{w_{k+1}}^2(x_{k+1},y_k) +  A_{k+1}\left( \sqrt{\frac{\mu^2\lambda_k^2}{(1+\mu\lambda_k)^2}+\frac{(1-\sigma_k)\mu\lambda_k}{2(1+\mu\lambda_k)}} - \frac{\mu\lambda_k}{1+\mu\lambda_k} \right) \cdot \\
    &\quad \left\| \mathtt{Exp}_{w_{k+1}}^{-1}(y_k)-\mathtt{Exp}_{w_{k+1}}^{-1}(x_{k+1}) \right\| \|v_{k+1}\| \\
    &\geq \frac{(1-\sigma_k)A_{k+1}}{4\lambda_k} d_{w_{k+1}}^2(x_{k+1},y_k) + \frac{(1-\sigma_k)A_{k+1}}{6} \sqrt{\frac{\mu\lambda_k}{1+\mu\lambda_k}} d_{w_{k+1}}(x_{k+1},y_k) \|v_{k+1}\|
    \end{aligned}
\end{equation}
where $\alpha, \beta$ and $\gamma$ are the coefficients defined in \Cref{rescale}. Hence, by \Cref{update_distortion_appendix} we have
\begin{equation}
    \notag
    \begin{aligned}
    p_k-p_{k+1} &\geq \frac{\mu}{2}(\theta_k a_{k+1}+A_k)\left\| \mathtt{Exp}_{w_{k+1}}^{-1}(y_k')-\mathtt{Exp}_{w_{k+1}}^{-1}(x_{k+1})+\mu^{-1}v_{k+1} \right\|^2 \\
    &\quad +\frac{A_{k+1}}{2\lambda_k\sigma_k}\left\| \mathtt{Exp}_{w_{k+1}}^{-1}(y_k)-\mathtt{Exp}_{w_{k+1}}^{-1}(x_{k+1})-\lambda_k v_{k+1} \right\|^2 \\
    &\quad + \frac{\mu\theta_{k}a_{k+1}A_k}{2(A_k+\theta_{k}a_{k+1})}d_{w_{k+1}}^2(x_k,z_k) - \frac{\sigma_k A_{k+1}}{2\lambda_k}d_{w_{k+1}}^2(x_{k+1},y_k)-\frac{A_{k+1}}{2\mu}\|v_{k+1}\|^2 \\
    &\geq \alpha d_{w_{k+1}}^2(x_{k+1},y_k) + 2\beta \left\langle \mathtt{Exp}_{w_{k+1}}^{-1}(y_k)-\mathtt{Exp}_{w_{k+1}}^{-1}(x_{k+1}), v_{k+1} \right\rangle + \gamma \|v_{k+1}\|^2 \\
    &\quad + \frac{\mu\theta_{k}a_{k+1}A_k}{2(A_k+\theta_{k}a_{k+1})}d_{w_{k+1}}^2(x_k,z_k) \\
    &\quad - \mu(\theta_k a_{k+1}+A_k) d_{w_{k+1}}(y_k,y_k') \cdot \left\| \mathtt{Exp}_{w_{k+1}}^{-1}(y_k)-\mathtt{Exp}_{w_{k+1}}^{-1}(x_{k+1})+\mu^{-1}v_{k+1} \right\| \\
    &\geq \frac{(1-\sigma_k)A_{k+1}}{4\lambda_k} d_{w_{k+1}}^2(x_{k+1},y_k) + \frac{\mu\theta_{k}a_{k+1}A_k}{2(A_k+\theta_{k}a_{k+1})}d_{w_{k+1}}^2(x_k,z_k) \\
    &\quad +\frac{(1-\sigma_k)A_{k+1}}{6} \sqrt{\frac{\mu\lambda_k}{1+\mu\lambda_k}} d_{w_{k+1}}(x_{k+1},y_k) \|v_{k+1}\| \\
    &\quad - \mu(\theta_k a_{k+1}+A_k) d_{w_{k+1}}(y_k,y_k') \cdot \left\| \mathtt{Exp}_{w_{k+1}}^{-1}(y_k)-\mathtt{Exp}_{w_{k+1}}^{-1}(x_{k+1})+\mu^{-1}v_{k+1} \right\|
    \end{aligned}
\end{equation}
as desired.
\end{proof}

In order to ensure potential decrease, it suffices to control the magnitude of the error term $d_{w_{k+1}}(y_k,y_k')$, as shown in the corollary below:
\begin{corollary}
 Suppose that
 \begin{equation}
     \notag
     d_{w_{k+1}}(y_k,y_k') \leq \frac{1-\sigma_k}{6}\min\left\{\sqrt{\mu\lambda_k},1 \right\} d_{w_{k+1}}(x_{k+1},y_k),
 \end{equation}
then we have the following inequality which implies potential decrease:
\begin{equation}
    \notag
    p_k-p_{k+1} \geq \frac{(1-\sigma_k)A_{k+1}}{12\lambda_k} d_{w_{k+1}}^2(x_{k+1},y_k) + \frac{\mu\theta_{k}a_{k+1}A_k}{2(A_k+\theta_{k}a_{k+1})}d_{w_{k+1}}^2(x_k,z_k).
\end{equation}
\end{corollary}

\begin{proof}
 Under the given condition, we can see that
 \begin{equation}
     \notag
     \begin{aligned}
     &\quad \frac{(1-\sigma_k)A_{k+1}}{6}\sqrt{\frac{\mu\lambda_k}{1+\mu\lambda_k}} d_{w_{k+1}}(x_{k+1},y_k) \\ 
     &\geq \frac{1-\sigma_k}{6}\sqrt{\mu\lambda_k(1+\mu\lambda_k)} (\theta_{k}a_{k+1}+A_k) \cdot \frac{6}{1-\sigma_k} \frac{1}{\sqrt{\mu\lambda_k}} d_{w_{k+1}}(y_k,y_k') \\
     &\geq (\theta_k a_{k+1}+A_k) d_{w_{k+1}}(y_k,y_k')
     \end{aligned}
 \end{equation}
 and
 \begin{equation}
     \notag
     \begin{aligned}
     \frac{(1-\sigma_k)A_{k+1}}{6\lambda_k} d_{w_{k+1}}(x_{k+1},y_k) \geq \frac{1+\mu\lambda_k}{\lambda_k} (\theta_k a_{k+1}+A_k) d_{w_{k+1}}(y_k,y_k') \geq \mu(\theta_k a_{k+1}+A_k) d_{w_{k+1}}(y_k,y_k')
     \end{aligned}
 \end{equation}
 Plugging the above inequalities into \eqref{rescale2_ineq}, we obtain the desired result.
\end{proof}

\begin{lemma}[restatement of \Cref{contraction}]
\label{contraction_appendix}
Suppose that $x\in\mathcal{M}$ and $y,a\in T_{x}\mathcal{M}$. Let $z = \mathtt{Exp}_x(a)$, then we have
\begin{equation}
\notag
d\left(\mathtt{Exp}_{x}(y+a), \mathtt{Exp}_{z}\left(\Gamma_{x}^{z} y\right)\right) \leq   \min \{\|a\|,\|y\|\} S_{K}(\|a\|+\|y\|)
\end{equation}
where
\begin{equation}
\notag
    S_{K}(r) =  \cosh\left(\sqrt{K}r\right)-\frac{\sinh\left(\sqrt{K}r\right)}{\sqrt{K}r}
\end{equation}
\end{lemma}

\begin{proof}
Define $\gamma(t) = \mathtt{Exp}_{x}(ta)$ and the curve
\begin{equation}
    \notag
    t \to c(r,t) = \mathtt{Exp}_{\gamma(t)}\left( r\Gamma_{x}^{\gamma(t)}(y+(1-t)a)\right)
\end{equation}
for fixed $r$. Let $J_t^{\text{norm}}(r) = \frac{\text{d}}{\text{d} t} c(r,t)$, then it is shown in ~\cite[Section B.3]{sun2019escaping} that 
\begin{equation}
    \notag
    d\left(\mathtt{Exp}_{x}(y+a), \mathtt{Exp}_{z}\left(\Gamma_{x}^{z} y\right)\right) \leq \int_{0}^{1} \left\| J_t^{\text{norm}}(1)\right\| \text{d} t
\end{equation}
Moreover, for fixed $t \in [0,1]$, let $\tilde{z} = \Gamma_{x}^{\gamma(t)}(y+(1-t)a)$ and $\rho_t = \|y+(1-t)a\| = \|\tilde{z}\|$, then its easy to see that $\|z\| \leq \|y\|+\|a\|$, and the proof in ~\cite[Section B.3]{sun2019escaping} implies that
\begin{equation}
    \notag
    \left\| J_t^{\text{norm}}(1)\right\| \leq \left\| J_t^{\text{norm}}(0)\right\| S_{K}(\rho_t) \leq \frac{\|a\|\cdot\|y\|}{\rho_t} S_{K}(\rho_t) \leq \frac{\|a\|\cdot\|y\|}{\|a\|+\|y\|} S_{K}(\|a\|+\|y\|),
\end{equation}
where the last step follows from the observation that $r^{-1}S_{K}(r)$ is increasing in $r$, by Taylor's expansion. Hence the result follows.
\end{proof}

The following lemma gives an upper bound of $d_{w_{k+1}}(y_k,y_k')$ in terms of function $S_k$ and a distance term.

\begin{lemma}[restatement of \Cref{addis_bound}]
\label{addis_bound_appendix}
We have for all $k \geq 1$ that 
\begin{equation}
    \notag
    \begin{aligned}
    d_{w_{k+1}}(y_k,y_k') \leq 2 d^{*}(w_{k+1};x_k,z_k)\cdot S_{K}\left( d(x_k,z_k)+d^*(w_{k+1};x_k,z_k) \right)
    \end{aligned}
\end{equation}
where $d^*(w;x,z) = \min\left\{ d(w,y): y = \mathtt{Exp}_{x}(t\cdot\mathtt{Exp}_{x}^{-1}(z), t \in [0,1] \right\}$ is the distance from $w$ to the geodesic connecting $x$ and $z$.
\end{lemma}

\begin{proof}
Let $\tau = \frac{\theta_{k}a_{k+1}}{A_k+\theta_k a_{k+1}}$, then by definition
\begin{equation}
    \notag
    y_k' = \mathtt{Exp}_{w_{k+1}}\left( (1-\tau)\mathtt{Exp}_{w_{k+1}}^{-1}(x_k)+\tau\mathtt{Exp}_{w_{k+1}}^{-1}(z_k) \right)
\end{equation}
Suppose that $w$ is a point on the geodesic connecting $x_k$ and $z_k$, then
\begin{equation}
\notag
    y_k = \mathtt{Exp}_{w}\left( (1-\tau)\mathtt{Exp}_{w}^{-1}(x_k)+\tau\mathtt{Exp}_{w}^{-1}(z_k) \right)
\end{equation}
We moreover define
\begin{equation}
    \notag
    y_k'' = \mathtt{Exp}_{w_{k+1}}\left( (1-\tau)\Gamma_{w}^{w_{k+1}}\mathtt{Exp}_{w}^{-1}(x_k)+\tau\Gamma_{w}^{w_{k+1}}\mathtt{Exp}_{w}^{-1}(z_k) + \mathtt{Exp}_{w_{k+1}}^{-1}(w) \right)
\end{equation}
Applying \Cref{contraction_appendix} with $x=w, z=w_{k+1}$ gives
\begin{equation}
    \notag
    d(y_k,y_k'') \leq d(w_{k+1},w)\cdot S_{K}\left( d(x_k,z_k)+2d(w_{k+1},w) \right)
\end{equation}
On the other hand,
\begin{equation}
    \begin{aligned}
    \notag
    d_{w_{k+1}}(y_k',y_k'') &\leq (1-\tau) d\left( x_k, \mathtt{Exp}_{w_{k+1}}\left( \Gamma_{w}^{w_{k+1}}\mathtt{Exp}_{w}^{-1}(x_k) +  \mathtt{Exp}_{w_{k+1}}^{-1}(w) \right)\right)  \\ 
    &\quad + \tau d\left( x_k, \mathtt{Exp}_{w_{k+1}}\left( \Gamma_{w}^{w_{k+1}}\mathtt{Exp}_{w}^{-1}(z_k) +  \mathtt{Exp}_{w_{k+1}}^{-1}(w) \right)\right) \\
    &\leq (1-\tau) d(w_{k+1},w)\cdot S_K\left( d(w,x_k)+2d(w,w_{k+1})\right) \\
    &\quad + \tau d(w_{k+1},w)\cdot S_K\left( d(w,z_k)+2d(w,w_{k+1})\right) \\
    &\leq d(w_{k+1},w)\cdot S_{K}\left( d(x_k,z_k)+2d(w,w_{k+1}) \right)
    \end{aligned}
\end{equation}
Combining the above inequalities, we obtain
\begin{equation}
    \notag
    d_{w_{k+1}}(y_k,y_k') \leq 2 d(w_{k+1},w)\cdot S_{K}\left( d(x_k,z_k)+2d(w,w_{k+1}) \right)
\end{equation}
The conclusion now follows from the definition of $d^*$.
\end{proof}

\begin{corollary}
Suppose that
\begin{equation}
    \label{suff_cond_potential_dec}
    12 d^*(w_{k+1};x_k,z_k) \cdot S_K\left(d(x_k,z_k)+2 d^*(w_{k+1};x_k,z_k)\right) \leq (1-\sigma_k) \min\left\{ \sqrt{\mu\lambda_k},1\right\} d_{w_{k+1}}(x_{k+1},y_k),
\end{equation}
then we have the following inequality which implies potential decrease:
\begin{equation}
    \label{potential_dec_with_addis}
    p_k-p_{k+1} \geq \frac{(1-\sigma_k)A_{k+1}}{12\lambda_k} d_{w_{k+1}}^2(x_{k+1},y_k) + \frac{\mu\theta_{k}a_{k+1}A_k}{2(A_k+\theta_{k}a_{k+1})}d_{w_{k+1}}^2(x_k,z_k).
\end{equation}
\end{corollary}

The corollary can be seen as a generalized version of the potential-decrease result we obtained in \Cref{cor_simple}. Indeed, when $w_{k+1}$ lies on the geodesic between $x_k$ and $z_k$, then the left hand side of \eqref{potential_dec_with_addis} equals zero, so that \eqref{potential_dec_with_addis} is guaranteed to hold.

We are now ready to prove our main result.

\begin{theorem}[formal version of \Cref{informal_convergence_general}]
\label{convergence_general}
Assume $f$ is $L$-smooth, and suppose that
\begin{itemize}
    \item $\sigma_k = \sigma \in (0,1)$.
    \item The sequence $\{w_k\}$ satisfies $d^2(w_k,x^*) \leq \omega \max\left\{ d(x_i,x^*), 0\leq i\leq k;  d(z_j,z^*), 0 \leq j\leq k-1 \right\}$.
    \item $d^*(w_{k+1};x_k,z_k) \leq \rho_1 d_{w_{k+1}}(x_{k+1},y_k)$ and $d^*(w_{k+1};x_k,z_k) \leq \rho_2 \max\left\{ d(x_k,x^*), d(z_k,x^*)\right\}$.
    \item The step size $\lambda_k=\lambda = \frac{c^2}{L}$ where $c\in (0,1)$ is a fixed constant.
    \item The initialization satisfies $d(x_0,x^*) \leq \frac{\tau}{20}K^{-\frac{1}{2}}\left(\frac{\mu}{L}\right)^{\frac{3}{4}}$ and $B_0 = \frac{\mu}{2}A_0 > 0$, where 
    \begin{equation}
        \notag
        \tau \leq \min\left\{ \sqrt{\frac{c}{2(2\omega+5)}},\sqrt{\frac{25(1-\sigma)c}{2\rho_1(7+10\rho_2^2)}}\right\}.
    \end{equation}
\end{itemize}
then for all $k \geq 0$, the following statements hold:
\begin{enumerate}[(1).]
    \item Potential decrease \eqref{potential_dec_with_addis} holds.
    \item $d^2(x_k,x^*) \leq \left( \frac{L}{\mu}+1\right) d^2(x_0,x^*) \leq \frac{\tau^2}{200}K^{-1}\left(\frac{\mu}{L}\right)^{\frac{1}{2}}$.
    \item The distortion rate $\delta_k \leq 1+ \frac{2\omega+5}{10}\tau^2 \sqrt{\frac{\mu}{L}}$.
    \item $\xi_k := \frac{a_k}{A_k} \geq \frac{9}{10}\sqrt{\frac{\mu\lambda}{1+\mu\lambda}}$ and hence $\frac{B_k}{A_k} \geq \frac{2}{5}\mu$.
    \item $d^2(z_k,x^*) \leq \frac{1}{80}K^{-1}\left(\frac{\mu}{L}\right)^{\frac{1}{2}}$.
\end{enumerate}
\end{theorem}

\begin{proof}
  We prove the result by induction on $k$. Specifically, for $k\geq 0$, we first prove (2),(3) and (5) hold for $k$, and then use them to derive (1),(4) for $k+1$, completing one round of induction step. When $k=0$, (2) follows from
  \begin{equation}
      \notag
      \delta_0 \leq 1+ 4K d^2(w_0,z_0) \leq 1+ 8K \left( d^2(w_0,x^*)+d^2(z_0,x^*) \right) \leq 1+ \frac{\omega+1}{50} \tau^2\sqrt{\frac{\mu}{L}},
  \end{equation}
  and the rest follows from the assumptions. Now suppose that the statements hold for $1,2,\cdots,k-1$. Consider the case for $k$. 
  
  The induction hypothesis implies that $d^2(x_k,x^*) \leq \frac{\tau^2}{200}K^{-1}\sqrt{\frac{\mu}{L}}$, and
  \begin{equation}
      \notag
      \begin{aligned}
      d_{w_k}^2(z_k,x^*) &\leq \frac{1}{B_k}p_0 \leq \frac{5}{2\mu}\frac{1}{A_0}\left( A_0(f(x_0)-f(x^*)+B_0 d_{w_0}^2(z_0,x^*) \right) \\
      &\leq \frac{5}{2\mu}\left( \frac{L}{2}d^2(x_0,x^*)+\frac{\mu}{2}d^2(x_0,x^*)\right) \\
      &= \frac{5}{4}\left( \frac{L}{\mu}+1 \right) d^2(x_0,x^*) \leq \frac{\tau^2}{160}K^{-1}\sqrt{\frac{\mu}{L}}
      \end{aligned}
  \end{equation}
  On the other hand, since
  \begin{equation}
      \label{wineq}
      d^2(w_k,x^*) \leq \omega d^2(x_k,x^*) \leq \frac{\omega\tau^2}{200}K^{-1}\left(\frac{\mu}{L}\right)^{\frac{1}{2}} < \frac{1}{2K},
  \end{equation}
  the distortion inequality \eqref{ahn_distortion} implies that
  \begin{equation}
      \label{zineq}
      d^2(z_k,x^*) \leq (1+4K d^2(w_k,x^*)) d_{w_k}^2(z_k,x^*) \leq \frac{\tau^2}{80}K^{-1}\sqrt{\frac{\mu}{L}}.
  \end{equation}
  The inequalities \eqref{wineq} and \eqref{zineq} together implies that
  \begin{equation}
      \notag
      d^2(z_k,w_k) \leq 2\left( d^2(w_k,x^*) + d^2(z_k,x^*)\right) \leq \frac{2\omega+5}{50}\tau^2 K^{-1}\sqrt{\frac{\mu}{L}}
  \end{equation}
  Hence, the distortion rate $\delta_k$ can be bounded as follows:
  \begin{equation}
      \notag
      \delta_k \leq 1 + 4K d^2(w_k,z_k) < 1+ \frac{2\omega+5}{10}\tau^2 \sqrt{\frac{\mu}{L}} \leq 1 + \frac{c}{20}\sqrt{\frac{\mu}{L}} \leq 1 + \frac{1}{10}\sqrt{\frac{\mu\lambda}{1+\mu\lambda}}.
  \end{equation}
  The induction hypothesis implies that $\xi_{k} \geq \frac{9}{10}\sqrt{\frac{\mu\lambda}{1+\mu\lambda}} =: \xi_{*}$, and recall the equation
  \begin{equation}
      \notag
      \delta_k\xi_{k+1}\left(\xi_{k+1}-\frac{\mu\lambda}{1+\mu\lambda}\right) = \xi_k^2(1-\xi_{k+1})
  \end{equation}
  To show $\xi_{k+1} \geq \frac{9}{10}\sqrt{\frac{\mu\lambda}{1+\mu\lambda}}$, it suffices to show that
  \begin{equation}
      \notag
      \delta_k \xi_{*} \left(\xi_{*}-\frac{\mu\lambda}{1+\mu\lambda}\right) \leq \xi_{*}^2 (1-\xi_{*}) \Leftrightarrow \delta_k \left( 1-\frac{10}{9}\xi_{*}\right) \leq 1-\xi_{*}
  \end{equation}
  The final equation holds since $\delta \leq 1 + \frac{1}{9}\xi_{*}$.
  
  Now it remains to show potential decrease $p_{k+1} \leq p_k$; it suffices to prove that \eqref{suff_cond_potential_dec} holds. Since $S_{K}(r) \leq \frac{1}{3}Kr^2$ when $Kr^2 \leq 1$, the assumptions imply that
  \begin{equation}
      \notag
      \begin{aligned}
      &\quad 12 d^*(w_{k+1};x_k,z_k) \cdot S_K\left(d(x_k,z_k)+2 d^*(w_{k+1};x_k,z_k)\right) \\
      &\leq 4\rho_1 d_{w_{k+1}}(x_{k+1},y_k) \cdot K \left(d(x_k,z_k)+2 d^*(w_{k+1};x_k,z_k)\right)^2 \\
      &\leq 4\rho_1 K \left( \frac{7}{100}\tau^2 K^{-1}+ \frac{1}{10}\rho_2^2\tau^2 K^{-1}\right)\sqrt{\frac{\mu}{L}}  d_{w_{k+1}}(x_{k+1},y_k) \\
      &\leq \frac{1-\sigma}{2} c \sqrt{\frac{\mu}{L}}  d_{w_{k+1}}(x_{k+1},y_k)
      \leq (1-\sigma) \sqrt{\mu\lambda} d_{w_{k+1}}(x_{k+1},y_k)
      \end{aligned}
  \end{equation}
  so that \eqref{suff_cond_potential_dec} holds. The proof is completed.
\end{proof}

Finally, we have the following corollary on acceleration for smooth functions.
\begin{corollary}
\label{cor_acc_appendix}
  Under the assumptions of \Cref{convergence_general}, we have
  \begin{equation}
      \notag
      f(x_k)-f(x^*) \leq \frac{1}{A_k} p_0 \leq \frac{\tau^2}{400}K^{-1}L\left(\frac{\mu}{L}\right)^{\frac{3}{2}} \left( 1- \frac{9\sqrt{c}}{10\sqrt{2}}\sqrt{\frac{\mu}{L}}\right)^{k}
  \end{equation}
\end{corollary}

\section{Details of \Cref{sec_1st_order}}
\label{1st_order_appendix}
In this section, we provide detailed description of the algorithms we discussed in \Cref{sec_1st_order} and verification that they can be recovered from the Riemannian A-HPE framework. Throughout this section, we assume that $f$ is $L$-smooth.

\subsection{Algorithms without the additional distortion}

First, we look at the \textit{Riemannian Nesterov's method}, which is proposed and studied in \citet{zhang2018estimate,ahn2020nesterov} and, to the best of our knowledge, the only provably accelerated method in our setting. The update of this method is given in \Cref{1st_RAGD}.

\begin{algorithm}
\SetKwInOut{KIN}{Input}
\caption{\textit{Riemannian} Nesterov's Method}
\label{1st_RAGD}
\KIN{Objective function $f$, initial point $x_0$,  $\sigma\in \left( 0,\frac{3}{4} \right)$, parameters $L,\mu$, initial weight $A_0 \geq 0$}
$z_0 \gets x_0$ and $\lambda \gets \frac{\sigma^2}{2L}$ \\
\For{$k=0,1,\cdots$}{
choose a valid distortion rate $\delta_k$ according to \Cref{ahn_distortion} \\
$\theta_{k} \gets$ the smaller root of $B_k (1-\theta)^2=\mu\lambda\theta\left( (1-\theta)B_k+\frac{\mu}{2}\delta_k A_k \right)$\\
$B_{k+1} \gets \frac{B_k}{\theta_k\delta_k}, a_{k+1} \gets 2\mu^{-1}(1-\theta_k)B_{k+1}$ and $A_{k+1} \gets A_k+a_{k+1}$ \\
$y_{k} \gets \mathtt{Exp}_{x_k}\left( \frac{\theta_k a_{k+1}}{A_k+\theta a_{k+1}}\mathtt{Exp}_{x_k}^{-1}(z_k) \right)$ \\
$x_{k+1} \gets \mathtt{Exp}_{y_k}(-\lambda\nabla f(y_k))$ \\
$z_{k+1} \gets \mathtt{Exp}_{y_k}\left( \theta_k\mathtt{Exp}_{y_k}^{-1}(z_k) - \mu^{-1}(1-\theta_k) \nabla f(y_k) \right)$\\
}
\end{algorithm}

\begin{proposition}
\Cref{1st_RAGD} can be recovered from \Cref{AHPE_Riemann} by choosing $\sigma_k = \sigma$, $\lambda \in \left(0, \frac{\sigma^2}{2L}\right)$, $w_{k+1} = y_k$, $x_{k+1} = \mathtt{Exp}_{y_k}(-\lambda_k\nabla f(y_k))$ and $v_{k+1}=\nabla f(y_k)+ \mu\mathtt{Exp}_{y_k}^{-1}(x_{k+1})$.
\end{proposition}

\begin{proof}
It remains to check that the specified update rule satisfies the inequality \Cref{Riemann_iprox} in the definition of \textit{iprox}. Indeed we have
\begin{equation}
\notag
    \begin{aligned}
    \text{LHS} &= \frac{\lambda_k}{2(1+\lambda_k\mu)}\left( f(x_{k+1})-f(y_k)-\left\langle \mathtt{Exp}_{y_k}^{-1}(x_{k+1}),\nabla f(y_k) \right\rangle\right) \\
    &\quad +\left(\frac{\lambda_k^2\mu^2}{2(1+\lambda_k\mu)^2}-\frac{\lambda_k\mu}{2(1+\lambda_k\mu)}\right)d^2(y_k,x_{k+1}) \\
    &\leq \frac{\lambda_k L}{2(1+\lambda_k\mu)}d^2(y_k,x_{k+1}) \leq \frac{\sigma_k^2}{2(1+\lambda_k\mu)^2}d^2(y_k,x_{k+1}) = \text{RHS}
    \end{aligned}
\end{equation}
so that the result follows.
\end{proof}

The second example is given in \Cref{3rd_RAGD}. It is a direct generalization of the accelerated method ~\cite[Algorithm 3]{chen2019first} to Riemannian setting, and can be viewed as a variant of Nesterov's method with an additional gradient descent step. To the best of our knowledge, the algorithm is new and its convergence property is not known in Riemannian setting.

\begin{algorithm}
\SetKwInOut{KIN}{Input}
\caption{\textit{Riemannian} Nesterov's method with an extra gradient step}
\label{3rd_RAGD}
\KIN{Objective function $f$, initial point $x_0$,  $\sigma\in \left( 0,\frac{3}{4} \right)$, parameters $L,\mu$, initial weight $A_0 \geq 0$}
$z_0 \gets x_0$ and $\lambda \gets \frac{\sigma^2}{2L}$ \\
\For{$k=0,1,\cdots$}{
choose a valid distortion rate $\delta_k$ according to \Cref{ahn_distortion} \\
$\theta_{k} \gets$ the smaller root of $B_k (1-\theta)^2=\mu\lambda\theta\left( (1-\theta)B_k+\frac{\mu}{2}\delta_k A_k \right)$\\
$B_{k+1} \gets \frac{B_k}{\theta_k\delta_k}$, $a_{k+1}\gets 2\mu^{-1}(1-\theta_k)B_{k+1}$ and $A_{k+1} \gets A_k+a_{k+1}$ \\
$x_k \gets \mathtt{Exp}_{\tilde{x}_k}(-\lambda\nabla f(\tilde{x}_k))$ \\
$y_{k} \gets \mathtt{Exp}_{x_k}\left( \frac{\theta_k a_{k+1}}{A_k+\theta a_{k+1}}\mathtt{Exp}_{x_k}^{-1}(z_k) \right)$ \\
$\tilde{x}_{k+1} \gets \mathtt{Exp}_{y_k}(-\lambda\nabla f(y_k))$ \\
$z_{k+1} \gets \mathtt{Exp}_{y_k}\left( \theta_k\mathtt{Exp}_{y_k}^{-1}(z_k) - \mu^{-1}(1-\theta_k) \nabla f(y_k) \right)$\\
}
\end{algorithm}

\begin{proposition}
\Cref{3rd_RAGD} can be recovered from \Cref{AHPE_Riemann} by choosing $\sigma_k=\sigma=\frac{3}{4}$, $w_{k+1}=y_k$, $v_{k+1}=\nabla f(y_k)+ \mu\mathtt{Exp}_{y_k}^{-1}(x_{k+1})$ and $x_{k+1}$ defined by
\begin{equation}
    \notag
    \tilde{x}_{k+1} = \mathtt{Exp}_{x_{k+1}}\left(-\lambda_k\nabla f(x_{k+1})\right),\quad x_{k+1}=\mathtt{Exp}_{\tilde{x}_{k+1}}\left(-\lambda_k\nabla f(\tilde{x}_{k+1}) \right).
\end{equation}
\end{proposition}

\begin{proof}
It suffices to check that the \textit{iprox} definition is satisfied. Smoothness implies that
\begin{equation}
    \notag
    f(x_{k+1})-f(y_k)-\left\langle \mathtt{Exp}_{y_k}^{-1}(x_{k+1}),\nabla f(y_k) \right\rangle \leq \frac{L}{2}d^2(x_{k+1},y_k).
\end{equation}
On the other hand, we can bound $\|\mathtt{Exp}_{y_k}^{-1}(x_{k+1})+\lambda_k v_{k+1}\|^2$ as follows:
\begin{equation}
    \label{3rd_RAGD_ineq}
    \begin{aligned}
    &\quad \|\mathtt{Exp}_{y_k}^{-1}(x_{k+1})+\lambda_k v_{k+1}\|^2
    = \| (1+\mu\lambda_k)\mathtt{Exp}_{y_k}^{-1}(x_{k+1})+\lambda_k\nabla f(y_k)\|^2 \\
    &= (1+\mu\lambda_k)^2 d^2(x_{k+1},y_k) +2\lambda_k(1+\mu\lambda_k)\left\langle \mathtt{Exp}_{y_k}^{-1}(x_{k+1}),\nabla f(y_k)\right\rangle+\lambda_k^2\|\nabla f(y_k)\|^2 \\
    &\leq (1+\mu\lambda_k)^2 d^2(x_{k+1},y_k)+\lambda_k^2\|\nabla f(y_k)\|^2 \\
    &\quad -2\lambda_k(1+\mu\lambda_k)\left( f(y_k)-f(x_{k+1})+\frac{\mu}{2}d^2(y_k,x_{k+1}) \right) \\
    &= (1+\mu\lambda_k)d^2(x_{k+1},y_k)+d^2(y_k,\tilde{x}_{k+1})  \\
    &\quad -2\lambda_k(1+\mu\lambda_k)\left( \frac{1}{\lambda_k}-\frac{L}{2}\right)\left( d^2(y_k,\tilde{x}_{k+1})+d^2(\tilde{x}_{k+1},x_{k+1}) \right) \\
    &\leq (1+\mu\lambda_k)d^2(x_{k+1},y_k)\\
    &\quad - 2\left( (1+\mu\lambda_k)\left(1-\frac{L}{2}\lambda_k\right)-\frac{1}{2} \right)\left( d^2(y_k,\tilde{x}_{k+1})+d^2(\tilde{x}_{k+1},x_{k+1}) \right) \\
    &\leq \left( \frac{L}{2}\lambda_k(1+\mu\lambda_k)+\frac{1}{2}\right)d^2(x_{k+1},y_k) 
    \end{aligned}
\end{equation}
Finally, the choice of $\lambda$ satisfies $L\lambda(1+\mu\lambda) \leq \frac{1}{2}$, hence the result follows.
\end{proof}

\subsection{Algorithms with the additional distortion}
In this section, we discuss specific examples of first-order methods that can be obtained from \Cref{AHPE_Riemann} as special cases. The setting considered here is more general than the previous subsection, in that we do not require that $w_{k+1}$ is chosen on the geodesic connecting $x_k$ and $z_k$, and we can apply \Cref{cor_acc} to obtain local (full) acceleration. 

We first present a method, called \textit{Riemannian accelerated extra-gradient descent (RAXGD)}, in \Cref{2nd_RAGD}. To see its difference with Riemannian Nesterov's method, note that it uses two gradients each iteration. RAXGD can be seen as a Riemannian and strongly-convex version of the \textit{accelerated extra-gradient method} proposed by \cite{diakonikolas2018accelerated}. To the best of our knowledge, this method has not been proposed or studied before.

\begin{algorithm}
\SetKwInOut{KIN}{Input}
\caption{\textit{Riemannian} accelerated extra-gradient descent}
\label{2nd_RAGD}
\KIN{Objective function $f$, initial point $x_0$, $\sigma_k\in \left( 0,1 \right)$, parameters $L,\mu$, initial weight $A_0, B_0 > 0$}
$z_0 \gets x_0$ and
$\lambda \gets \frac{\sigma}{L}$\\
\For{$k=0,1,\cdots$}{
choose a valid distortion rate $\delta_k$ according to \Cref{ahn_distortion} \\
$\theta_{k} \gets$ the smaller root of $B_k (1-\theta)^2=\mu\lambda_k\theta\left( (1-\theta)B_k+\frac{\mu}{2}\delta_k A_k \right)$\\
$B_{k+1} \gets \frac{B_k}{\theta_k\delta_k}, a_{k+1} = 2\mu^{-1}(1-\theta_k)B_{k+1}$ and $A_{k+1}=A_k+a_{k+1}$  \\
$y_{k} \gets \mathtt{Exp}_{x_k}\left( \frac{\theta_k a_{k+1}}{A_k+\theta a_{k+1}}\mathtt{Exp}_{x_k}^{-1}(z_k) \right)$ \\
$x_{k+1} \gets \mathtt{Exp}_{y_k}(-\lambda\nabla f(y_k))$\\
$z_{k+1} \gets \mathtt{Exp}_{x_{k+1}}\left( \theta_k\mathtt{Exp}_{x_{k+1}}^{-1}(z_k) - \mu^{-1}(1-\theta_k) \nabla f(x_{k+1})\right)$\\
}
\end{algorithm}

The following proposition shows that \Cref{2nd_RAGD} can be considered as a special case of \Cref{AHPE_Riemann}.

\begin{proposition}
\Cref{2nd_RAGD} can be recovered from \Cref{AHPE_Riemann} by choosing $\sigma_k=\sigma \in(0,1)$, $\lambda \leq \frac{\sigma}{L}$, $v = \nabla f(x_{k+1})$ and $w_{k+1} = x_{k+1} = \mathtt{Exp}_{y_k}(-\lambda_k\nabla f(y_k))$. Moreover, the conditions in \Cref{convergence_general} are satisfied with $\rho_1 = \rho_2 = \omega = 1$.
\end{proposition}

\begin{proof}
We have $(x_{k+1},v_{k+1}) \in\mathtt{iprox}_{f}^{w_{k+1}}(y_k,\lambda_k,\varepsilon_k)$, since
\begin{equation}
\notag
\begin{aligned}
\left\| \mathtt{Exp}_{x_{k+1}}^{-1}(y_k)-\lambda_k\nabla f(x_{k+1}) \right\| &= \lambda_k \left\| \Gamma_{y_k}^{x_{k+1}}\nabla f(y_k)-\nabla f(x_{k+1}) \right\|\\
&\leq L\lambda_k d(y_k,x_{k+1}) \leq \sigma_k d(x_{k+1},y_k).
\end{aligned}
\end{equation}
Finally, note that
\begin{equation}
    \notag
    d^*(w_{k+1};x_k,z_k) \leq d(x_{k+1},y_k) \leq \frac{1}{L}\left\|\nabla f(y_k)\right\| \leq d(y_k,x^*),
\end{equation}
the conclusion follows.
\end{proof}

We can also design new accelerated algorithms by choosing different realizations of the \textit{iprox} operator in \Cref{AHPE_Riemann}. This can lead to novel algorithms that are previously unknown even in Euclidean setting. In the following we derive from \Cref{AHPE_Riemann} a generalized version of RAXGD, given in \Cref{new_RAGD}.

\begin{algorithm}
\SetKwInOut{KIN}{Input}
\caption{Generalized Riemannian accelerated extra-gradient descent}
\label{new_RAGD}
\KIN{Objective function $f$, initial point $x_0$, $\sigma_k\in \left( 0,1 \right)$, parameters $L,\mu$, initial weight $A_0, B_0 > 0$}
$z_0 \gets x_0$ and
$\lambda \gets \frac{\sigma}{2L}$\\
\For{$k=0,1,\cdots$}{
choose a valid distortion rate $\delta_k$ according to \Cref{ahn_distortion} \\
$\theta_{k} \gets$ the smaller root of $B_k (1-\theta)^2=\mu\lambda_k\theta\left( (1-\theta)B_k+\frac{\mu}{2}\delta_k A_k \right)$\\
$B_{k+1} \gets \frac{B_k}{\theta_k\delta_k}, a_{k+1} = 2\mu^{-1}(1-\theta_k)B_{k+1}$ and $A_{k+1}=A_k+a_{k+1}$  \\
$y_{k} \gets \mathtt{Exp}_{x_k}\left( \frac{\theta_k a_{k+1}}{A_k+\theta a_{k+1}}\mathtt{Exp}_{x_k}^{-1}(z_k) \right)$ \\
$w_{k+1} \gets \mathtt{Exp}_{y_k}(-\lambda\nabla f(y_k))$\\
choose $x_{k+1}$ such that $d(w_{k+1},x_{k+1})\leq \frac{1-\sigma}{3}d(y_k,x_{k+1})$ \\
$z_{k+1} \gets \mathtt{Exp}_{x_{k+1}}\left( \theta_k\mathtt{Exp}_{x_{k+1}}^{-1}(z_k) - \mu^{-1}(1-\theta_k) \nabla f(x_{k+1})\right)$\\}
\end{algorithm}

Rather than obtain $x_{k+1}$ directly from a gradient descent step, it allows arbitrary choices of $x_{k+1}$ as long as a distance inequality
\begin{equation}
    \label{new_RAGD_ineq}
    d(w_{k+1},x_{k+1})\leq \frac{1-\sigma}{3}d(y_k,x_{k+1})
\end{equation}
is satisfied. The inequality is obviously satisfied when $x_{k+1} = w_{k+1}$, which reduces to \Cref{2nd_RAGD}. Intuitively, \Cref{new_RAGD_ineq} implies that $x_{k+1}$ is obtained by starting from $y_k$ and following an `approximately descent' direction. In Euclidean setting, the solution set of \Cref{new_RAGD_ineq} for $x_{k+1}$ is a region enclosed by an Apollonius circle that contains $w_{k+1}$.

\begin{proposition}
\Cref{new_RAGD} is a special case of \Cref{AHPE_Riemann}. Moreover, the conditions in \Cref{convergence_general} holds with $\rho_1 = 4, \rho_2 = 1$ and $\omega = \frac{3}{2}$.
\end{proposition}

\begin{proof}
To check that \Cref{new_RAGD} can be obtained from \Cref{AHPE_Riemann}, it suffices to verify that the update of $x_{k+1}$ satisfies \Cref{Riemann_iprox}.

Since $f$ is $L$-smooth, we have
\begin{equation}
    \notag
    \frac{\lambda}{1+\mu\lambda}\left(f(x_{k+1})-f_{w_{k+1}}(x_{k+1})\right) \leq \frac{L\lambda}{2}d^2(x_{k+1},w_{k+1}) \leq \frac{1}{2} d^2(x_{k+1},y_k).
\end{equation}
On the other hand 
\begin{equation}
    \notag
    \begin{aligned}
    &\quad \|\mathtt{Exp}_{w_{k+1}}^{-1}(x_{k+1})-\mathtt{Exp}_{w_{k+1}}^{-1}(y_k)+\lambda v_{k+1}\|^2 \\
    &= \|(1+\mu\lambda)\mathtt{Exp}_{w_{k+1}}^{-1}(x_{k+1})-\mathtt{Exp}_{w_{k+1}}^{-1}(y_k) + \lambda\nabla f(w_{k+1})\|^2 \\
    &\leq 2(1+\mu\lambda)^2 d^2(w_{k+1},x_{k+1}) + 2\lambda^2 \left\| \nabla f(w_{k+1}) - \Gamma_{y_k}^{w_{k+1}} \nabla f(y_k)\right\|^2 \\
    &\leq 2(1+\mu\lambda)^2 d^2(w_{k+1},x_{k+1}) + 2L^2\lambda^2 d^2(w_{k+1},y_k) \\
    &\leq 2\left( (1+\mu\lambda)^2 + 2L^2\lambda^2\right) d^2(w_{k+1},x_{k+1}) + 4L^2\lambda^2 d^2(x_{k+1},y_k) \\
    &\leq 3(1+\mu\lambda)^2 d^2(w_{k+1},x_{k+1}) + \sigma^2 d^2(x_{k+1},y_k) \\
    &\leq (1+\mu\lambda)^2 d^2(x_{k+1},y_k)
    \end{aligned}
\end{equation}
where the last step uses \Cref{new_RAGD_ineq}. Hence \Cref{Riemann_iprox} holds.
\end{proof}

By \Cref{convergence_general}, we deduce that \Cref{new_RAGD} achieves acceleration when initialized in an $\mathcal{O}\left( K^{-\frac{1}{2}}\left(\frac{\mu}{L}\right)^{\frac{3}{4}}\right)$. To the best of our knowledge, \Cref{new_RAGD} has not been studied even for strongly-convex functions in the Euclidean setting.

We emphasize that the purpose of introducing \Cref{new_RAGD} is to show that \Cref{AHPE_Riemann} can lead to many different types of accelerated first-order methods. There are of course other ways to specify the \textit{iprox} operator, which would lead to many interesting algorithms. 

\subsection{Discussion of the extra-point framework in \citet{huang2021unifying}}
\label{comparison}

In a recent work \citet{huang2021unifying}, the authors propose an extra-point approach motivated by the analysis of classical accelerated methods. Based on this idea, they propose a framework for smooth strongly-convex optimization, which is in a quite general form and contains a total of $9$ parameters. For convenience we give the detailed updates of their framework below.

\begin{subequations}\label{huang_framework}
    \begin{align}
    p_k &\gets t_1 x_k + t_2 z_k \label{p_k}\\
    y_k &\gets \text{a solution of } \left\langle\nabla f(y_k), p_k-y_k\right\rangle \geq 0 \label{y_k}\\
    \widetilde{x}_{k+1} &\gets y_k - \frac{t_3}{L} \nabla f(y_k) \label{xline_1}\\
    x_{k+1} &\gets y_k - \frac{t_4}{L} \nabla f(\widetilde{x}_{k+1}) - \frac{t_5}{L}\left( \nabla f(\widetilde{x}_{k+1}) - \nabla f(y_k) \right) + t_6 \left( \widetilde{x}_{k+1} - y_k \right) \label{xline_2} \\
    z_{k+1} &\gets t_7 z_k + t_8 y_k - t_9 \nabla f(y_k) \label{z}
    \end{align}
\end{subequations}

The authors derive sufficient conditions on the choice of $t_i, 1\leq i\leq 9$ so that \eqref{huang_framework} can achieve acceleration. While their framework looks complicated, in the following we show that it can be interpreted quite naturally from the PPM viewpoint introduced in \Cref{section_euclidean}.

First, \Cref{z} is very similar to the update of $z_{k+1}$ in A-HPE; one can see this by comparing it with \Cref{alternative_update}, with the choice $w_{k+1}=y_k$ and $v_{k+1} = \nabla f(y_k) + \mu\left( x_{k+1}-y_k\right)$. With properly chosen constants $t_7,t_8,t_9$, \Cref{z} can then be interpreted as an approximate PPM scheme.

Second, \Cref{xline_1} and \Cref{xline_2} together give a gradient-descent-type update formula of $x_{k+1}$. In particular, \Cref{xline_2} can also be written as
\begin{equation}
    \notag
    x_{k+1} \gets y_k - \frac{t_3 t_6-t_5}{L}\nabla f(y_k) - \frac{t_4+t_5}{L}\nabla f(\widetilde{x}_{k+1}),
\end{equation}
which is very similar to the Riemannian Nesterov's method with multiple gradient steps that we introduced in \Cref{3rd_RAGD}. As a result, the update of $x_{k+1}$ can also be interpreted as another approximate PPM scheme.

Recall the arguments in \Cref{section_euclidean} that the final step is to combine these two schemes and obtain potential decrease. In A-HPE this is implemented by a simple convex combination of the iterates $x_k$ and $z_k$. However, in \Cref{huang_framework} the procedure is more complex: first a convex combination is obtained (i.e. the update of $p_k$), and then $y_k$ is chosen to be any solution of the inequality \Cref{y_k}.

This procedure, in fact, can be easily justified by one additional step in the analysis: intuitively, $y_k$ is a \textit{refinement} of the convex combination. Specifically, as argued in the proof of \Cref{potential_dec}, the combination of two PPM approaches is implemented by the following inequality:
\begin{equation}
\notag
    \theta_z \|z_k-x_{k+1}+\mu^{-1}v_{k+1}\|^2 + \theta_x \|x_k-x_{k+1}+\mu^{-1}v_{k+1}\|^2 \geq (\theta_z+\theta_x)\|p_k-x_{k+1}+\mu^{-1}v_{k+1}\|^2.
\end{equation}
Since $x_{k+1}-\mu^{-1}v_{k+1} = y_k-\mu^{-1}\nabla f(y_k)$, we have
\begin{equation}
    \notag
    \|p_k-x_{k+1}+\mu^{-1}v_{k+1}\|^2 = \|p_k-y_k+\mu^{-1}\nabla f(y_k)\|^2 \geq \|\nabla f(y_k)\|^2 = \|y_k-x_{k+1}+\mu^{-1}v_{k+1}\|^2,
\end{equation}
where the inequality exactly follows from \Cref{y_k}! 

Now we have seen that the framework of \citep{huang2021unifying} uses the same idea of approximate-PPM as A-HPE, except that the combination step is more general. On the other hand, the framework is limited to the choice of $w_{k+1} = y_k$ in the definition of \textit{iprox}, while A-HPE allows more flexible choices.

Finally, we provide a natural extension of the framework to the Riemannian setting:
\begin{equation}
    \notag
    \begin{aligned}
    p_k &\gets \mathtt{Exp}_{x_k}\left( t_2 \mathtt{Exp}_{x_k}^{-1}(z_k)\right) \\
    y_k &\gets \text{a solution of } \left\langle \nabla f(y_k), \mathtt{Exp}_{y_k}^{-1}(p_k)\right\rangle \geq 0 \\
    \widetilde{x}_{k+1} &\gets \mathtt{Exp}_{y_k}\left( -\frac{t_3}{L}\nabla f(y_k)\right) \\
    x_{k+1} &\gets \mathtt{Exp}_{\widetilde{x}_{k+1}}\left( (1-t_6)\mathtt{Exp}_{\widetilde{x}_{k+1}}^{-1}(y_k) - \frac{t_4}{L}\nabla f(\widetilde{x}_{k+1}) - \frac{t_5}{L}\left( \nabla f(\widetilde{x}_{k+1})-\Gamma_{y_k}^{\widetilde{x}_{k+1}}\nabla f(y_k)\right)\right) \\
    z_{k+1} &\gets \mathtt{Exp}_{y_k}\left( t_7 \mathtt{Exp}_{y_k}^{-1}(z_k) - t_9\nabla f(y_k)\right)
    \end{aligned}
\end{equation}

Local acceleration of the framework can be shown using our in  approach \Cref{section_riemann}. For the special case $y_k=p_k$, the additional distortion disappears and the framework attains global eventual acceleration.


\end{document}